\documentclass[11pt,dvipsnames]{amsart}

\usepackage{xcolor}
\usepackage{multirow}
\definecolor{myblue}{HTML}{1E3A5F}
\usepackage{amsmath,amssymb,amsfonts}
\usepackage[colorlinks,linkcolor=myblue,citecolor=myblue,urlcolor=myblue]{hyperref}
\usepackage{amsthm}
\usepackage[foot]{amsaddr}
\usepackage[nameinlink]{cleveref}
\Crefname{equation}{}{}
\crefname{lemma}{Lemma}{Lemmas}
\Crefname{lemma}{Lemma}{Lemmas}
\usepackage[utf8]{inputenc}
\usepackage{graphicx}
\usepackage{color}
\usepackage{float}
\usepackage{soul}
\usepackage{lipsum}
\usepackage{comment}
\usepackage{bm}
\usepackage{mathtools}
\usepackage{etoolbox}
\usepackage{algorithm}
\usepackage{algpseudocode}
\usepackage{fancybox}
\usepackage{enumitem}
\usepackage{booktabs}
\usepackage{stmaryrd}
\usepackage{a4wide}

\graphicspath{{../}}

\newtheorem{theorem}{Theorem}[section]
\newtheorem{lemma}[theorem]{Lemma}
\newtheorem{corollary}[theorem]{Corollary}

\newtheorem{remark}[theorem]{Remark}
\newtheorem{assumption}[theorem]{Assumption}
\newtheorem{example}[theorem]{Example}

\Crefname{assumption}{Assumption}{Assumptions}

\DeclareMathOperator*{\argmin}{arg\,min}
\DeclareMathOperator*{\argmax}{arg\,max}

\DeclareFontFamily{U}{matha}{\hyphenchar\font45}
\DeclareFontShape{U}{matha}{m}{n}{
<-6> matha5 <6-7> matha6 <7-8> matha7
<8-9> matha8 <9-10> matha9
<10-12> matha10 <12-> matha12
}{}
\DeclareSymbolFont{matha}{U}{matha}{m}{n}
\DeclareFontFamily{U}{mathx}{\hyphenchar\font45}
\DeclareFontShape{U}{mathx}{m}{n}{
<-6> mathx5 <6-7> mathx6 <7-8> mathx7
<8-9> mathx8 <9-10> mathx9
<10-12> mathx10 <12-> mathx12
}{}
\DeclareSymbolFont{mathx}{U}{mathx}{m}{n}
\DeclareMathDelimiter{\vvvert}{0}{matha}{"7E}{mathx}{"17}

\DeclarePairedDelimiterX{\normiii}[1]{\vvvert}{\vvvert}{\ifblank{#1}{\:\cdot\:}{#1}}

\makeatletter
\newcommand{\tnorm}{\@ifstar\@tnorms\@tnorm}
\newcommand{\@tnorm}[2][]{%
	\mathopen{#1|\mkern-1.5mu#1|\mkern-1.5mu#1|}
	#2
	\mathclose{#1|\mkern-1.5mu#1|\mkern-1.5mu#1|}
}
\makeatother

\definecolor{LightRed}{rgb}{0.9,0.55,0.5}

\newcommand{\sfp}{\mathsf{p}}
\newcommand{\R}{\mathbb{R}}
\newcommand{\fa}{\text{ for all }}
\newcommand{\dd}{\operatorname{d}\!}
\newcommand{\DG}{\mathrm{DG}}

\newcommand{\bfn}{\boldsymbol{n}}

\title[Proximal Discontinuous Galerkin Methods]{Proximal Discontinuous Galerkin Methods for Variational Inequalities}

\author{Alexandre Ern$^1$$^2$, Brendan Keith$^3$, Dohyun Kim$^3$, Rami Masri$^3$, Beatrice Riviere$^4$}
\address{$^1$ CERMICS, CNRS, ENPC, Institut Polytechnique de Paris, Marne-la-Vall\'ee cedex~2, F-77455, France}
\address{$^2$ Centre Inria de Paris, 48 rue Barrault, F-75647, Paris, France}
\email{alexandre.ern@enpc.fr}
\address{$^3$ Division of Applied Mathematics, Brown University, Providence, RI~02912}
\email{brendan\_keith@brown.edu, dohyun\_kim@brown.edu}
 \email{rami\_masri@brown.edu}
\address{$^4$ Department of Computational Applied Mathematics and Operations Research,
Rice University, Houston, TX~77005, USA}
\email{riviere@rice.edu}
\thanks{BK, DK, and RM were supported in part by the U.S.\ Department of Energy,
Office of Science Early Career Research Program under Award Number DE-SC0024335
and by the Center for Information Geometric Mechanics and Optimization (CIGMO),
a PSAAP-IV Focused Investigatory Center funded by the U.S.\ Department of Energy,
National Nuclear Security Administration under Award Number DE-NA0004261.
BK was also supported in part by the Alfred P.\ Sloan Foundation via a Sloan Research
Fellowship in Mathematics.
BR is partially supported by NSF-DMS~2513092.
The authors are thankful to Z.\ Dong (INRIA Paris) for his assistance in the
implementation of the HHO method.}

\begin{document}

\begin{abstract}
We introduce a family of proximal discontinuous Galerkin methods for variational inequalities, focusing on the obstacle problem as a didactic example. Each member of this family is born from applying a different well-known nonconforming finite element discretization to the Bregman proximal point method. We explicitly treat four examples: the symmetric interior penalty discontinuous Galerkin, the enriched Galerkin, the hybridizable interior penalty and the hybrid high-order methods. We formulate a unified analysis framework for this family of methods and prove the existence and uniqueness of solutions, energy dissipation, and error estimates for both the primal and dual variables. Remarkably, the proximal hybrid high-order method with piecewise constant cell unknowns and piecewise affine facet unknowns leads to the first higher-order convergence result for any proximal Galerkin method.

\vspace{1em}
\smallskip

\end{abstract}
\maketitle

\section{Introduction}

Variational problems with pointwise inequality constraints arise throughout applied mathematics and engineering, encompassing contact mechanics, elastoplasticity, and phase-field models~\cite{Glowinski1984}. The unilateral Poisson obstacle problem, which models the equilibrium state of an elastic membrane constrained to lie above a given obstacle, is a canonical example and the focus of the present work.

The proximal Galerkin (PG) method~\cite{keith2023proximal} is a finite element framework that combines the Bregman proximal point method~\cite{chen1993convergence} with the finite element method, and acts simultaneously as an iterative algorithm and a discretization scheme. The PG method exhibits mesh-independent iteration counts, meaning that the number of nonlinear iterations to reach a prescribed accuracy remains bounded independently of the mesh size~\cite{papadopoulos2024hierarchical, dokken2025latent, keith2025priori}. In fact, the error analysis for PG methods shows that the optimization error is decoupled from the discretization error \cite{keith2025priori}.  The PG framework has been applied to multiple free-boundary problems, fracture modeling, gradient constraints, topology optimization, and the enforcement of discrete maximum principles~\cite{keith2023proximal, dokken2025latent,kim2025simple,papadopoulos2024hierarchical}. Most recently, the PG framework has been extended to non-symmetric variational inequalities, including American option pricing~\cite{fu2026proximal}. 

Most of the studied PG methods rely on \textit{conforming} finite element spaces \cite{keith2023proximal,dokken2025latent,kim2025simple}. One exception is the first-order system proximal Galerkin method (FOSPG) proposed for elliptic obstacle problems in \cite{fu2026locally} and then extended to non-symmetric advection--diffusion variational inequalities in \cite{fu2026proximal}. The latter results in a nonconforming\ PG method with upwinding. Convergence guarantees for the conforming PG methods are found in \cite{keith2025priori} and for the FOSPG methods in \cite{fu2026locally,fu2026proximal}. While the framework incorporates high-order discretization schemes naturally \cite{keith2023proximal,dokken2025latent}, we note that all error analyses of the method to date have been limited to establishing first-order spatial convergence rates. 

In this paper, we further extend the PG framework to \emph{nonconforming} finite element discretizations, focusing on the unilateral Poisson obstacle problem for simplicity. 
All the discretization schemes rely on broken finite element spaces and loosely fit into the broad category of discontinuous Galerkin (DG) methods. Specifically, we consider here four instances of such methods. \textup{(i)} The symmetric interior penalty DG (IPDG) method analyzed for elliptic problems in \cite{arnold2002unified} (see also \cite{Baker:77,Wheeler:78} for seminal works and \cite{riviere2008discontinuous,di2011mathematical} for two textbooks on the topic). 
\textup{(ii)} The enriched Galerkin (EG) method, originally introduced in \cite{Ischia:03} and further developed in, e.g., \cite{Sun_Liu:09,Lee_Lee_Wheeler:16}.
\textup{(iii)} The hybridizable interior penalty (H-IP) DG method, which can be traced back to \cite{Lab_Wel:07} (called therein interface stabilized finite element method) and \cite{Oikawa:10} (where a lifting operator is additionally employed in the stabilization), but has since been further developed in, e.g., \cite{Rhe_Wells:17,EtFaFR:22,KirkRiviereMasri2023}.
\textup{(iv)} The hybrid high-order (HHO) method, which was introduced in \cite{DiPEL:14} for linear diffusion and in \cite{DiPEr:15} for locking-free linear elasticity problems. Note that HHO was bridged to hybridizable DG methods in \cite{CoDPE:16}. For a gentle introduction to HHO methods, we refer the reader to the textbook \cite{CiErPi:21}. 
The four above discretization methods can be grouped into two sub-classes based on the presence of facet variables. In particular, IPDG and EG solely rely on discrete cell unknowns living in a broken polynomial space, whereas the H-IP and HHO methods additionally introduce facet unknowns living in a suitable broken polynomial space.
Finally, we notice that we do not explicitly consider DG methods involving flux variables that approximate $-\nabla u$, such as the local DG method \cite{LDG_98}. Many such methods can be recast in an abstract primal formulation by means of discrete lifting operators, following \cite{arnold2002unified}. Likewise, hybridized flux formulations, such as HDG methods \cite{cockburn2009unified}, can be interpreted as HHO methods \cite{CoDPE:16}. Consequently, the framework developed in this paper applies to DG methods with flux formulations as well.

\subsection{Main contributions} Our main contributions can be summarized as follows: 
\begin{itemize}
    \item We introduce the class of proximal discontinuous Galerkin (DG) methods and unify its error analysis into a common abstract framework. This framework is applicable beyond the four specific methods that we consider herein (IPDG, EG, H-IP and HHO).

    \item We generalize the \textit{a priori} error analysis in \cite{keith2025priori} to the nonconforming setting. In particular, we show best-approximation-type results and error rates in a suitable energy norm. The error bound accounts for the optimization error and demonstrates that this error is decoupled from the spatial discretization error.
    
    \item We prove unique solvability of each nonlinear subproblem (\Cref{thm:existence}), monotone energy dissipation (\Cref{lemma:energy}), and error estimates for both primal and dual variables under minimal regularity assumptions (\Cref{thm:bestapprox} and \Cref{thm:multiplier_convergence}), with improved estimates under mild additional regularity assumptions (\Cref{thm:general_error}).

\item For the case of the proximal HHO method with piecewise constant cell unknowns and piecewise affine facet unknowns, we demonstrate that the locally reconstructed solution yields an energy error estimate of spatial order $h^{\frac32 -\epsilon}$ for $\epsilon >0$. This is the first rigorous higher-order convergence result for any PG method.
\end{itemize}

\subsection{Outline} The remainder of the paper is organized as follows. The model problem and notation are introduced in \Cref{sec:model}. \Cref{sec:Bregman_proximal} presents the continuous-level Bregman proximal point method from which the PG framework derives. The four proximal DG methods are introduced in a unified template in \Cref{sec:PDG}. \Cref{sec:error_analysis} develops the abstract error analysis, including well-posedness, energy dissipation, best approximation, and convergence error rates. \Cref{sec:specific_methods} applies the abstract framework to each specific discretization. 
The paper ends with numerical experiments verifying the theoretical predictions in Section~\ref{sec:num} followed by some brief concluding remarks in Section~\ref{sec:conc}.

\section{Model problem}\label{sec:model}
Let $\Omega\subset\mathbb{R}^n$, $n \in \{2,3\}$, be an open, bounded Lipschitz domain. Throughout this article, we use the standard notation for the Sobolev--Hilbert spaces $H^m(\Omega)$, $m \geq 1$, $m \in \mathbb N$. For non-integer $s$, the notation $H^{s}(\Omega)$ denotes the Sobolev--Slobodeckij spaces. The space $V:=H^1_0(\Omega)$ is composed of functions in $H^1(\Omega)$ with vanishing trace on the boundary of $\Omega$; i.e., $\operatorname{tr}(v) = 0 \operatorname{ on } \partial \Omega$.  The dual space of $H^1_0(\Omega)$ is denoted by $H^{-1}(\Omega)$. We use the notation $(\cdot, \cdot)$ to denote the $L^2(\Omega)$ inner product. 
  
We focus the presentation on the following unilateral Poisson obstacle problem:
\begin{equation}\label{eq:objective}
\min_{v \in K} E(v), \quad E(v):= \frac{1}{2} \|\nabla v\|_{L^2(\Omega)}^2 -  (f, v) ,   
\end{equation}
where $f \in L^2(\Omega)$ is given and the feasible set $K$ is
\begin{equation}
K := \{ v \in V \mid v \geq  \phi  \text{ a.e. in } \Omega \},
\end{equation}
where $\phi \in H^1(\Omega) \cap C^0(\overline \Omega)$ with $\phi \leq 0$ on $\partial \Omega$. The set $K$ is a nonempty, closed and convex subset of $V$.
We define the bilinear form  $a(\cdot,\cdot): H^1_0(\Omega) \times H^1_0(\Omega) \to \mathbb{R}$ and the linear form $\ell: L^2(\Omega) \to \mathbb{R}$ as
\begin{equation}
    a(u,v) := (\nabla u, \nabla v), \quad \ell(v) := (f, v).  
\end{equation} 
The solution $u \in K$ to \eqref{eq:objective} is uniquely characterized by the following variational inequality:
\begin{equation}\label{eq:vi}
    a(u,v-u)-\ell(v-u)\geq 0  \quad ~\fa v\in K.
\end{equation}
The unique Lagrange multiplier is denoted by $\lambda \in H^{-1}(\Omega)$ and satisfies 
\begin{align}\label{eq:conti-lambda}
\langle \lambda, v \rangle  = a(u,v) - \ell(v) ~ \fa v \in H^1_0(\Omega).
 \end{align}
In the sequel, we require the following set of observables:
\begin{equation}
    \mathcal{O} := \{ o \in L^2(\Omega) \mid o(x) \in [\phi(x), \infty) \text{ for almost every } x \in \Omega \}.
 \label{eq:ConstrainedObservables}
 \end{equation}
 
\section{The Bregman proximal point method}\label{sec:Bregman_proximal} 
This section introduces the necessary tools from convex analysis to define the Bregman proximal point method for the unilateral Poisson obstacle problem \eqref{eq:objective}. The nonconforming methods that we present and later analyze are a discretization of this method. 
\subsection{Legendre functions} \label{subsec:Legendre} A key component of the proximal Galerkin framework is the functional $\mathcal{R}^*$, which maps $L^\infty(\Omega)$ to $\mathcal{O}$ via its gradient. This functional is constructed using the notion of Legendre functions \cite{RTRockafellar_1970}. 
In this context, we define a Legendre function $L: \R \rightarrow \R \cup \{+\infty\}$ as a proper, strictly convex function that is differentiable on the interior of its essential domain, $\operatorname{int} (\operatorname{dom} L) \neq \emptyset$, with $\operatorname{dom} L := \{y \in \mathbb R \mid L(y) < +\infty \}$, and whose gradient is singular only at the boundary of $\operatorname{dom} L$.

Let $R: \Omega \times \R  \rightarrow \R \cup \{+ \infty\}$ be a Carath\'eodory function such that, for almost every $x \in \Omega$, the map $ R(x, \cdot)$ is a Legendre function with essential domain  $[\phi(x), \infty)$. 
We define the associated superposition operator
\begin{equation}
    \mathcal{R}(u)(x) := R(x,u(x)), \quad \forall x \in \Omega, \,  \forall u \in \mathcal{O}, \label{eq:Legendere_0}
\end{equation}    
whose gradient is given by 
\begin{equation}
\nabla \mathcal{R}(u)(x) = \partial_u R(x,u(x)). \label{eq:grad_legendre}
\end{equation} 
Note that $\operatorname{dom}(\nabla \mathcal{R}) = \operatorname{int} (\operatorname{dom}(\mathcal R)).$
The convex conjugate of $R(x,\cdot)$ determines the conjugate superposition operator $\mathcal{R}^*$ as follows:
\begin{equation}\label{eq:convex_conj_S}
    R^*(x,z)
    := \sup_{y \in \mathbb{R}}  \big\{ zy - R(x,y) \big\}.
\end{equation} 
Throughout this work, we assume that $R(x,\cdot)$ is supercoercive,  $R(x, y)/|y| \rightarrow \infty$ as $|y| \rightarrow \infty$ for almost every $x \in \Omega$.
This assumption guarantees that $R^*(x,\cdot)$ is both well-defined and continuously differentiable on $\R$ \cite[Proposition 2.16]{bauschke1997legendre}; see also \cite[Corollary 13.3.1]{RTRockafellar_1970}.
Furthermore, both $R(x,\cdot)$ and $R^*(x,\cdot)$ are essentially strictly convex on $\mathbb{R}$.
The superposition operator is $\mathcal{R}^*(\psi)(x):=R^*(x,\psi(x))$ with gradient given by $\nabla \mathcal{R}^*(\psi)(x) = \partial_{z} R^*(x , \psi(x)) $. The gradient $\nabla \mathcal{R}^*$ is strictly monotone and continuous on $L^\infty(\Omega)$, so that
\begin{equation} \label{eq:monotonicity}
    (\nabla \mathcal{R}^*(\psi) - \nabla \mathcal{R}^*(\xi), \psi - \xi) \geq 0 \quad \forall \psi, \xi \in L^\infty(\Omega),
\end{equation}
with equality if and only if $\psi = \xi$.
We recall the fundamental property of Legendre functions established in \cite{rockafellar1967conjugates}:
\begin{equation} \label{eq:conj_grad_inverse}
 \nabla \mathcal{R}^*  = (\nabla \mathcal{R})^{-1}.
\end{equation} 
This relation implies that $\operatorname{dom}(\nabla \mathcal R) = \operatorname{im}(\nabla \mathcal R^*) \subset \mathcal{O}$. In particular, we have $ \nabla \mathcal R^* (\psi)(x) \in (\phi(x), \infty)$ for almost every $x \in \Omega$.
Explicit examples are provided in \Cref{example:shannon_entropy,example:softplus}; additional cases can be found in \cite[Table 1]{dokken2025latent}. We focus on unilateral constraints here, but the framework easily accommodates bilateral constraints by incorporating the Fermi–Dirac entropy \cite[Section~5.1]{keith2023proximal}.

Using the identity \eqref{eq:conj_grad_inverse}, we can introduce a latent representation for any observable variable $u \in \operatorname{dom}(\nabla \mathcal{R}) \subset \mathcal{O}$. Specifically, we define the latent variable $\psi$ via the relation:
\begin{equation}
\psi = \nabla \mathcal{R} ( u) \iff \nabla \mathcal{R}^*(\psi) = u.
\label{eq:latent_primal}
\end{equation}

\begin{example}[Shannon entropy] \label{example:shannon_entropy}
Define
\[ 
R(x,y) :=\begin{cases} (y - \phi(x)) \ln (y - \phi(x)) - (y -\phi(x))&\text{ when }y> \phi(x),\\+\infty&\text{ otherwise.}
\end{cases}
\]
The corresponding superposition operator is 
\begin{equation*}
\mathcal{R}(u) := (u -\phi) \ln (u -\phi) - (u - \phi). 
\end{equation*}
We deduce that  $\nabla \mathcal{R}(u) = \ln (u - \phi)$ whenever 
\[ 
u \in \operatorname{dom}(\nabla \mathcal{R}) = \{ u \in L^{\infty}(\Omega) \mid \operatorname{ess} \operatorname{inf} (u - \phi) > 0 \}. 
\]
A simple computation shows that $R^*(x,z) = \exp(z) + \phi(x) z$ so that
\begin{equation}
    \mathcal{R}^*(\psi) = \exp(\psi) + \phi \psi, \quad \quad \nabla \mathcal{R}^* (\psi) = \exp(\psi) + \phi.   \label{eq:shannon}
\end{equation}
Note that $\nabla \mathcal{R}^*$ is well defined on $L^{\infty}(\Omega)$, strictly monotone, and $\nabla \mathcal{R}^*(\psi) > \phi$ a.e.\ in $\Omega$.
\end{example}

\begin{example}[Softplus] \label{example:softplus}
Define
\[
R(x,y) := \begin{cases} (y - \phi(x)) \ln(e^{y-\phi(x)} - 1) + \operatorname{Li}_2(1 - e^{y-\phi(x)})  &\text{ when } y > \phi(x), \\ +\infty & \text{ otherwise, } \end{cases}
\]
where $\operatorname{Li}_2(z) = -\int_0^z \frac{\ln(1-t)}{t}\,\mathrm{d}t$ is the dilogarithm.
The corresponding superposition operator satisfies $\nabla \mathcal{R}(u) = \ln(e^{u-\phi} - 1)$ whenever 
\[
u\in \operatorname{dom}(\nabla \mathcal{R}) = \{ u \in L^\infty(\Omega) \mid \operatorname{ess\,inf}(u - \phi) > 0 \}.
\]
A direct computation shows that $R^*(x,z) = \phi(x) z - \operatorname{Li}_2(-e^z)$, so that
\begin{equation}
    \mathcal{R}^*(\psi) = \phi\,\psi - \operatorname{Li}_2(-e^\psi), \quad \nabla \mathcal{R}^*(\psi) = \phi + \ln(1 + e^\psi). \label{eq:softplus}
\end{equation}
Note that $\nabla \mathcal{R}^*$ is well defined on $L^\infty(\Omega)$, strictly monotone, and $\nabla \mathcal{R}^*(\psi) > \phi$ a.e.\ in $\Omega$.
\end{example}

\subsection{Bregman divergence}
The Legendre functions defined in \Cref{subsec:Legendre} allow one to define the Bregman divergence, which is the next key component of the proximal Galerkin framework \cite[Section~4.3]{keith2023proximal}.  For $u \in \mathrm{dom}(\mathcal{R})$ and $v \in \mathrm{dom}(\nabla \mathcal{R})$, the Bregman divergence is given by the error in the following first-order Taylor expansion:
\begin{equation}
    \label{eq:Bregman}
    \mathcal{D}(u,v) := \mathcal{R}(u) - \mathcal{R}(v) - \nabla \mathcal{R}(v)(u - v) . 
\end{equation}
Recalling that $R(x,\cdot)$ is strictly convex,
$\mathcal{D}(u,v) \geq 0$ a.e.~in $\Omega$, and equality is attained if and only if $u=v$.
We also recall the three-point identity \cite[Lemma 3.1]{chen1993convergence}:
\begin{align}
\mathcal{D}(u,v) - \mathcal{D}(u,w) + \mathcal{D}(v,w) & = (\nabla \mathcal{R}(v) -\nabla \mathcal{R}(w)) ( v- u).   \label{eq:three_point_iden}
\end{align}

\subsection{Bregman proximal point method}\label{subsec:BPP}
Given the Bregman divergence \eqref{eq:Bregman}, the Bregman proximal point method is an iterative method to approximate the solution to the unilateral Poisson obstacle problem~\eqref{eq:objective}. The method reads as follows \cite{chen1993convergence,keith2023proximal}: Given $u^{0} \in \operatorname{dom}(\nabla \mathcal{R})$ and an unsummable sequence of positive proximity parameters $(\alpha_k)_{k\geq 1}$, find
\begin{equation}\label{eq:BPP}
    u^k := \operatorname*{arg\,min}_{v \in K} \left\{ \alpha_k E(v) + \int_{\Omega} \mathcal{D}(v, u^{k-1}) \, \mathrm{d}x \right\}, \qquad k = 1,2,\ldots
\end{equation}
Since $E$ is strongly convex and $\mathcal{D}(\cdot, u^{k-1})$ is strictly convex, \eqref{eq:BPP} always has a unique solution. In fact, one can also demonstrate that $u^k \in \operatorname{dom}(\nabla \mathcal{R})$ \cite[Appendix A.3]{keith2023proximal}. Setting the latent variable $\psi^k = \nabla \mathcal{R}(u^k)$, the optimality conditions together with the identity $\nabla \mathcal{R}^* = (\nabla \mathcal{R})^{-1}$ from \eqref{eq:conj_grad_inverse} provide an equivalent algorithm: Given $\psi^{0} \in L^\infty(\Omega)$ and an unsummable sequence of positive proximity parameters $(\alpha_k)_{k\geq 1}$, find $u^k \in V:=H^1_0(\Omega)$ and $\psi^k \in W:=L^\infty(\Omega)$ such that
\begin{subequations}\label{eq:BPP_saddle}
\begin{alignat}{2}
     a(u^k, v) + \frac{1}{\alpha_k}(\psi^k - \psi^{k-1}, v) &= \ell(v) &&~\fa v \in V, \label{eq:BPP_saddle1}\\
     (u^k, w) - ( \nabla \mathcal{R}^*(\psi^k) , w) &= 0 &&~\fa w \in W. \label{eq:BPP_saddle2}
\end{alignat}
\end{subequations}
We refer the interested reader to \cite{keith2023proximal} for a detailed analysis of the above method applied to the unilateral Poisson obstacle problem, establishing well-posedness and convergence in $V$ under additional assumptions. The proximal DG methods introduced below are obtained by discretizing \eqref{eq:BPP_saddle} with nonconforming spaces $V_h \not \subset V$. 

\section{The proximal discontinuous Galerkin methods}
\label{sec:PDG}

We introduce various proximal DG methods for the unilateral Poisson obstacle problem \eqref{eq:objective}. We begin by presenting the required notation before briefly summarizing each method. The methods are subsequently presented in a general and unified manner.

\subsection{Notation}
Let $\mathcal{T}_h$ be a conforming shape-regular simplicial partition of $\Omega$.
We consider two finite-dimensional vector spaces, $V_h$ and $W_h$. In general, the space $V_h$ is nonconforming, namely $V_h \not \subset V$, while the space $W_h$ is conforming, namely $W_h \subset L^\infty(\Omega)$. Note that we allow $V_h$ to be a product space when considering hybrid discretizations.
We denote by $\mathcal{F}_h$ (resp.\ $\mathcal{F}_h^\partial$) the set of interior (resp.\ boundary) facets of $\mathcal{T}_h$.
For each interior face $F$, we fix a unit normal vector $\bfn_F$ and denote by $T_-$ and $T_+$ the neighboring elements such that $\bfn_F$ points from $T_-$ to $T_+$. The jump and average operators are canonically defined by:
\[
[v]|_F  = v|_{T_-} - v|_{T_+}, \quad 
\{v\}|_F = \frac12 (v|_{T_-} + v|_{T_+}), \quad F = \partial T_- \cap \partial T_+.
\]
By convention, if $F$ is a boundary face, then the unit normal vector $\bfn_F$ is chosen to point outward from $\Omega$, whereas the jump and average of a function on $F$ both reduce to its trace.  In what follows, we drop the subscript $F$ from the jump and average operators when $F$ is clear from the context. 
For each triangle $T$, the vector $\bfn_T$ is the unit normal vector pointing outward from $T$. 
The $L^2$ inner-product on an element $T$ (resp.\ face $F$) is denoted by $(\cdot,\cdot)_T$  (resp.\ $(\cdot,\cdot)_F$). 
We denote by $h_T$ the diameter of an element $T \in \mathcal{T}_h$, by $h_F$ the diameter of a face $F \in \mathcal{F}_h$, and by $h := \max_{T \in \mathcal{T}_h} h_T$ the mesh size.

We will make use of the broken $H^1$-space: 
    \begin{equation}
    H^1(\mathcal{T}_h) := \{v \in L^2(\Omega) \mid v \vert_T \in H^1(T)\;\forall T \in \mathcal{T}_h \}.   
    \end{equation}
We define the classical DG norm on the space $H^1(\mathcal{T}_h)$ as 
\begin{equation} 
\label{eq:DGnorm}
\|v\|_{\DG}^2 := \sum_{T \in \mathcal{T}_h}\|\nabla v\|_{L^2(T)}^2 + \sum_{F \in \mathcal{F}_h \cup \mathcal{F}_h^\partial} h_F^{-1}\|[v]\|_{L^2(F)}^2, \quad \forall v \in H^1(\mathcal{T}_h).
\end{equation}
For every integer $p \geq 0$, the broken polynomial space on $\mathcal{T}_h$ is defined as
\begin{equation}
\mathbb{P}^p(\mathcal{T}_h) := \{ v \in L^2(\Omega) \mid v|_T \in \mathbb{P}^p(T) \; \forall T \in \mathcal{T}_h \},
\end{equation}
and the facet polynomial space with zero values on boundary facets is defined as
\begin{equation}
\mathbb{P}^p(\mathcal{F}_h) := \{ \mu \in L^2(\mathcal{F}_h \cup \mathcal{F}_h^\partial) \mid \mu|_F \in \mathbb{P}^p(F) \; \forall F \in \mathcal{F}_h \text{ and } \mu \vert_F = 0 \; \forall F \in \mathcal{F}_h^\partial \}.
\end{equation}
For each element $T \in \mathcal{T}_h$, we denote the set of faces of $T$ by $\mathcal{F}_T$ and define the local facet polynomial space
\begin{equation}
\mathbb{P}^p(\mathcal{F}_T) := \{ \mu \in L^2(\partial T) \mid \mu|_F \in \mathbb{P}^p(F) \; \forall F \in \mathcal{F}_T \}.
\end{equation}

\subsection{Proximal IPDG and EG}\label{sec:dG_intro}
We begin with the symmetric interior penalty discontinuous Galerkin (IPDG) method. The vector spaces for the primal and latent variables are given, respectively, by
\begin{equation}
V_h := V_{\mathcal{T}_h}:= \mathbb{P}^1(\mathcal{T}_h), \qquad W_h := \mathbb{P}^0(\mathcal{T}_h).
\end{equation} 
The symmetric IPDG bilinear form $a_h: V_h \times V_h \rightarrow \R$ is written explicitly in~\eqref{eq:ipdg_form}. 
Given $\psi_h^{0} \in W_h$ and an unsummable sequence of positive proximity parameters $(\alpha_k)_{k \geq 1}$, the proximal IPDG method reads: find $u_h^k \in V_h$ and $\psi_h^k \in W_h$ such that
\begin{subequations}
\begin{alignat}{2} 
 a_h(u_h^k, v_h) + \frac{1}{\alpha_k}( v_h , \psi_h^k - \psi_h^{k-1}) &=  (f, v_h), \quad &&\forall v_h \in V_h, \label{eq:dG_intro_1} \\
( u_h^k , w_h)
- (\nabla \mathcal{R}^\ast(\psi_h^k), w_h) &= 0, \quad
&&\forall w_h \in W_h. \label{eq:dG_intro_2}
\end{alignat}
\end{subequations}
The enriched Galerkin (EG) method enriches the conforming Lagrange element with piecewise constants. It uses the same bilinear form as IPDG, but with the spaces  
\begin{equation}
V_h := V_{\mathcal{T}_h}:= (\mathbb{P}^1(\mathcal{T}_h) \cap H_0^1(\Omega)) + \mathbb{P}^0(\mathcal{T}_h),
\qquad W_h := \mathbb{P}^0(\mathcal{T}_h).
\end{equation} 

\begin{remark}[Non-symmetric variants]
Non-symmetric or incomplete interior pe\-nalty DG can also be incorporated. However, since $a_h(\cdot,\cdot)$ is no longer symmetric, there is no corresponding discrete energy. Therefore, these variants require a different analysis approach, see \cite{fu2026proximal}, where error estimates for the weighted averages of the iterates can be obtained for the conforming and FOSPG discretizations.  The analysis for non-symmetric variants in the nonconforming setting is beyond the scope of this paper.
\end{remark}

\subsection{Proximal H-IP and HHO methods}\label{sec:hybrid_intro}
We also consider the class of hybrid methods, in which $V_h$ is a product space composed of element unknowns and facet unknowns. Specifically, we set for $\ell,r\in\{0,1\}$, 
\begin{equation}
V_h := V_{\mathcal{T}_h}\times V_{\mathcal{F}_h} 
:= \mathbb{P}^\ell(\mathcal{T}_h) \times \mathbb{P}^r(\mathcal{F}_h), \qquad
W_h := \mathbb{P}^0(\mathcal{T}_h).
\end{equation}
A generic function in $V_h$ is denoted by $v_h := (v_{\mathcal{T}_h}, v_{\mathcal{F}_h})$.
The hybridizable interior penalty (H-IP) and hybrid high-order (HHO) methods both fall into this category with $(\ell,r)=(1,1)$ and $(\ell,r)\in\{(0,0),(0,1),(1,1)\}$, respectively.

Given $\psi_h^{0} \in W_h$ and an unsummable sequence of positive proximity parameters $(\alpha_k)_{k \geq 1}$, the proximal hybrid DG methods read: find $u_h^k := (u_{\mathcal{T}_h}^k, u_{\mathcal{F}_h}^k) \in V_h$ and $\psi_h^k \in W_h$ such that 
\begin{subequations}
\begin{alignat}{2} 
 a_h(u_h^k, v_h) + \frac{1}{\alpha_k}( v_{\mathcal{T}_h} , \psi_h^k - \psi_h^{k-1}) &=  (f, v_{\mathcal{T}_h}), \quad &&\forall v_h \in V_h, \label{eq:hybrid_intro_1}\\
( u_{\mathcal{T}_h}^k , w_h)
-(\nabla \mathcal{R}^\ast(\psi_h^k), w_h) &= 0, \quad
&&\forall w_h \in W_h. \label{eq:hybrid_intro_2}
\end{alignat}
\end{subequations}
The specific bilinear forms for the H-IP and HHO methods are provided in \Cref{sec:HIP,sec:HHO}, respectively.

\begin{remark}[Further HHO variant]
It is also possible to consider the polynomial parameters $(\ell,r)=(1,0)$ for the HHO method. This choice leads to a simpler form of stabilization, but yields the same error estimates as the choice $(\ell,r)=(0,0)$. 
\end{remark}

\subsection{Unified presentation of the proximal DG methods} 

We unify the presentation of the algorithm for the above approximation schemes, allowing us to present a unified analysis framework in the next section.

In all cases, a nonconforming subspace $V_{\mathcal{T}_h}\subset H^1(\mathcal{T}_h)$ is involved in the approximation of the primal variable, and the space for approximating the latent variable is always $W_h:=\mathbb{P}^0(\mathcal{T}_h)$. We notice that we have
\begin{equation}
\mathbb{P}^0(\mathcal{T}_h) \subseteq V_{\mathcal{T}_h} \subseteq \mathbb{P}^1(\mathcal{T}_h).
\end{equation}
In what follows, we consider a linear map
\begin{equation}\label{eq:sfp}
\sfp : V_h \rightarrow V_{\mathcal{T}_h}, \qquad
W_h \subseteq \sfp(V_h).
\end{equation}
For the IPDG and EG methods, $V_{\mathcal{T}_h}=V_h=\mathbb{P}^1(\mathcal{T}_h)$ and $\sfp$ is simply the identity. For the H-IP and HHO methods, we set $\sfp(v_h) := v_{\mathcal{T}_h}$ for all $v_h := (v_{\mathcal{T}_h}, v_{\mathcal{F}_h}) \in V_h$, so that $V_{\mathcal{T}_h} = \mathbb{P}^r(\mathcal{T}_h)$ with $r=1$ for H-IP and $r\in\{0,1\}$ for HHO. In all cases, the inclusion $W_h \subseteq \sfp(V_h)$ holds true.

We always consider a symmetric bilinear form $a_h: V_h \times V_h \rightarrow \R$ for simplicity. Explicit expressions are given in \Cref{sec:specific_methods}, and abstract assumptions are stated in \Cref{sec:error_analysis}.
For clarity,~\Cref{tab:methods} lists the equation reference corresponding to $a_h$ for each method considered.

With the above ingredients, the abstract proximal DG method reads as follows:
Find $u_h^{k} \in  V_{h}$ and $\psi_h^k \in W_h$ such that 
\begin{subequations} \label{eq:discrete_lvpp_bis}
\begin{alignat}{2} 
a_h(u_h^k, v_h) + \frac1{\alpha_k}(\sfp(v_h), \psi_h^k - \psi_h^{k-1}) & =  (f, \sfp(v_h)), \quad &&  \forall v_h \in V_h, \label{eq:genPG1_bis}\\
(\sfp(u_h^k), w_h) -(\nabla \mathcal{R}^\ast(\psi_h^k), w_h) & = 0, \quad && \forall w_h \in W_h.\label{eq:genPG2_bis}
\end{alignat}
\end{subequations}
In practice, we adopt the equivalent reformulation given in \Cref{alg:main_alg_discrete}.
Experience shows that formulation~\eqref{eq:discrete_lvpp} is more convenient for implementation, while~\eqref{eq:discrete_lvpp_bis} is more convenient for analysis.

\begin{algorithm}[htb]
\caption{The Proximal DG Method} 
\begin{algorithmic}[1
]\label{alg:main_alg_discrete}
    \State \textbf{input:} Initial latent solution guess $\psi_h^0  \in W_h$, an unsummable sequence of positive proximity parameters $(\alpha_k)_{k\geq1}$, and a strictly convex functional $\mathcal{R}^*$ with $\nabla \mathcal{R}^*: W_h \rightarrow \mathcal{O}$.
    \State \textbf{repeat}
    \State \quad Find $u_h^{k} \in  V_{h}$
and $\lambda_h^k \in W_h$ such that 
\begin{subequations} \label{eq:discrete_lvpp}
\begin{alignat}{2} 
 a_h(u_h^k, v_h) - (\sfp(v_h), \lambda_h^k)& =  (f, \sfp(v_h)), \quad &&  \forall v_h \in V_h, \label{eq:genPG1}\\
(\sfp(u_h^k), w_h)
-(\nabla \mathcal{R}^\ast(\psi_h^{k-1} - \alpha_k\lambda_h^k), w_h) & = 0, \quad && 
\forall w_h \in W_h.
\label{eq:genPG2}
\end{alignat}
\end{subequations} 
\State \quad Assign  \(k \gets k + 1\), $\psi_h^k = \psi_h^{k-1} - \alpha_k\lambda_h^k$.
    \State \textbf{until} a convergence test is satisfied.
\end{algorithmic}
\end{algorithm}
\begin{table}[!ht]
\centering
\caption{Summary of proximal DG methods.}
\label{tab:methods}
\renewcommand{\arraystretch}{1.6}
\begin{tabular}{@{\hskip 4pt}l@{\hskip 8pt}|@{\hskip 8pt}l@{\hskip 8pt}l@{\hskip 8pt}l@{\hskip 4pt}}
\toprule
 & & $V_h$ & $a_h$ \\
\midrule
\multirow{2}{*}{{Primal}} & IPDG & $\mathbb{P}^1(\mathcal{T}_h)$ & \eqref{eq:ipdg_form} \\
 & EG & $(\mathbb{P}^1(\mathcal{T}_h) \cap H_0^1(\Omega)) + \mathbb{P}^0(\mathcal{T}_h)$ & \eqref{eq:ipdg_form} \\
\midrule
\multirow{2}{*}{{Hybrid}} & H-IP & $\mathbb{P}^1(\mathcal{T}_h) \times \mathbb{P}^1(\mathcal{F}_h)$ & \eqref{eq:hip_form} \\
 & HHO & $\mathbb{P}^\ell(\mathcal{T}_h) \times \mathbb{P}^r(\mathcal{F}_h), \, \ell,  r \in \{0,1\}, \,\, \ell \leq r$ & \eqref{eq:hho_form} \\
\bottomrule
\end{tabular}
\end{table}

The proximal DG method provides a bound-preserving approximation to the true solution $u$, which we denote as
\begin{align}
o_h^k := \nabla \mathcal{R}^*(\psi_h^k).
\end{align}
By~\cref{eq:conj_grad_inverse}, this solution variable always satisfies the pointwise bound constraint $o_h^k(x) \in (\phi(x), \infty)$. Moreover,
since $\nabla \mathcal{R}^\ast = (\nabla \mathcal{R})^{-1}$ by \eqref{eq:conj_grad_inverse}, we infer that
\begin{equation}\label{eq:gradoh}
\psi_h^k = \nabla \mathcal{R}(o_h^k). 
\end{equation}
Furthermore, we can define the discrete Lagrange multiplier variable as 
\begin{align}
\lambda_h^k := \frac{\psi_h^{k-1} - \psi_h^k}{\alpha_k} \in W_h \subset L^2(\Omega) \hookrightarrow H^{-1}(\Omega).
\end{align}
The last embedding means that we interpret $\lambda_h^k \in H^{-1}(\Omega)$ via the duality pairing 
\begin{equation} 
\langle \lambda_h^k , v \rangle := (\lambda_h^k, v), \quad \forall v \in V.
\end{equation}

\begin{remark}[SPD structure of hybrid methods]
\label{rem:StaticCondensation}
For the hybrid methods, the structure of the stiffness matrix at each Newton iteration can be written as
\begin{equation*}
    \begin{bmatrix}
    A&B^T&-M^T\\B&C&\\M& &\alpha H
    \end{bmatrix},
\end{equation*}
where $H$ corresponds to the derivative of the second term in \eqref{eq:genPG2} and $\alpha$ is the current proximity parameter.
The two upper-left blocks stem from the discretization of the primal variable.
The methods we consider admit a symmetric positive definite (SPD) Schur complement $S:=-BA^{-1}B^{\mathrm{T}}+C$.
Since $\mathcal{R}^*$ is strictly convex and $W_h=\mathbb{P}^0(\mathcal{T}_h)$, $H$ is block-diagonal with SPD blocks. Therefore, we can locally eliminate cell degrees of freedom. The resulting global stiffness matrix $\tilde{S}$, involving only the facet degrees of freedom, is defined as $\tilde{S}:=-B\tilde{A}^{-1}B^{\mathrm{T}}+C$ where $\tilde{A}:=A+\frac{1}{\alpha}M^{\mathrm{T}}H^{-1}M,$ so that $\tilde{S}$ is again an SPD matrix.
\end{remark}

\section{Error analysis for proximal DG methods}
\label{sec:error_analysis}
In this section, we provide a unified error analysis of the proximal DG methods introduced in the previous section. 

\subsection{Basic setting}

We start by formulating some basic assumptions on the approximation setting, which we refine later on when we deal with the error analysis. We equip the discrete space $V_h$ with a norm $\|\cdot\|_{V_h}$ and we equip the discrete space $W_h$ with the $L^2$-norm. 
We make the following assumptions.

\begin{assumption}[Basic setting] \label{ass:basic}
The following holds: \textup{(i)} The discrete bilinear form $a_h$ is coercive and bounded:
\begin{subequations}
\begin{alignat}{2}
&a_h(v_h,v_h) \geq C_{\mathrm{coerc}} \|v_h\|^2_{V_h}, &\quad &\forall v_h \in V_h, \label{eq:coercivity_a} \\
&a_h(v_h,z_h) \leq C_{\mathrm{cont}} \Vert v_h \Vert_{V_h}\, \Vert z_h \Vert_{V_h}, &\quad &\forall v_h, z_h \in V_h.\label{eq:continuity_a}
\end{alignat}
\textup{(ii)} The linear map $\sfp:V_h\rightarrow V_{\mathcal{T}_h}$ is bounded: There is a constant $C_{\sfp}$ (uniform in $h$) so that
\begin{equation} \label{eq:bnd_sfp}
\|\sfp(v_h)\|_{\DG} \le C_{\sfp} \|v_h\|_{V_h}, \quad \forall v_h\in V_h.
\end{equation}
\textup{(iii)} There exists a linear map $\Pi_h: H^1(\mathcal{T}_h) \rightarrow V_h$ and a constant $C_{\Pi}$ (uniform in $h$) so that, for all $v\in H^1(\mathcal{T}_h)$,
\begin{equation}\label{eq:Pihprop}
(J_h v-v,w_h)=0 \; \forall w_h \in W_h \qquad \text{and} \qquad
\Vert  \Pi_h v \Vert_{V_h} \leq C_{\Pi} \|v\|_{\DG},
\end{equation} 
where we have set $J_h:=\sfp\circ \Pi_h : H^1(\mathcal{T}_h) \rightarrow V_{\mathcal{T}_h}$.
\end{subequations}
\end{assumption}

\begin{remark}[Map $J_h$]
In all the nonconforming methods considered herein, the linear map $J_h$ turns out to be
the $L^2$-orthogonal projection onto $V_{\mathcal{T}_h}$. We keep, however, an abstract notation for the sake of generality.
\end{remark}

\subsection{Existence, uniqueness and energy dissipation} \label{sec:basic_pty}

We start with some fundamental results on the proximal DG method: existence and uniqueness of the iterates and energy dissipation.

\begin{theorem}[Existence and uniqueness] For all $k \geq 1$, the system \eqref{eq:discrete_lvpp} has a unique solution $(u_h^k, \psi_h^k) \in V_h \times W_h$.  
\label{thm:existence}
\end{theorem}
\begin{proof}
The proof closely follows \cite[Theorem 3.1]{keith2025priori}.  We provide some details for the nonconforming setting for completeness. 
We first prove uniqueness. For a fixed $k$, assume that $u_h^k,\psi_h^k$ and $\tilde{u}_h^k, \tilde{\psi}_h^k$ are two solution pairs. We have by coercivity (see~\eqref{eq:coercivity_a})
\begin{align*}
C_{\mathrm{coerc}} \Vert u_h^k-\tilde{u}_h^k \Vert_{V_h}^2
& \leq a_h(u_h^k-\tilde{u}_h^k, u_h^k-\tilde{u}_h^k)
= -\frac{1}{\alpha_k} (\sfp(u_h^k-\tilde{u}_h^k), \psi_h^k-\tilde{\psi}_h^k)\\
& = -\frac{1}{\alpha_k} (\nabla \mathcal{R}^\ast(\psi_h^k)-\nabla \mathcal{R}^\ast(\tilde{\psi}_h^k), \psi_h^k-\tilde{\psi}_h^k) \leq 0,
\end{align*}
where the first equality follows from~\eqref{eq:genPG1_bis}, the second from~\eqref{eq:genPG2_bis}, and the last bound from the monotonicity of $\nabla \mathcal{R}^\ast$ (see \eqref{eq:monotonicity}). This implies that $u_h^k=\tilde{u}_h^k$. Then \eqref{eq:genPG1} yields 
\[
(\sfp(v_h), \psi_h^k-\tilde{\psi}_h^k) = 0, \quad \forall v_h\in V_h.
\]
To conclude that $\psi_h^k-\tilde{\psi}_h^k = 0$, we recall that $W_h \subseteq \sfp(V_h)$ (see~\eqref{eq:sfp}). 

Next, we prove existence. We define the Lagrangian $\mathcal{L}: V_h\times W_h \rightarrow \mathbb{R}$ so that
\[
\mathcal{L}(v,w) = \frac{\alpha_k}{2} a_h(v,v)
-\alpha_k (f, \sfp(v)) + (\sfp(v),w-\psi_h^{k-1}) - (\mathcal{R}^\ast(w),1),
\]
and observe that to prove existence, it suffices to find
a critical point of $\mathcal{L}$. For this purpose, we first
minimize $\mathcal{L}$ in the first variable. For any fixed $w\in W_h$, we solve for $\tilde{v} = \tilde{v}(w) \in V_h$ that satisfies
\[
\tilde{v}(w) = \argmin_{v\in V_h} \mathcal{L}(v,w).
\]
This is equivalent to solving for $\tilde v \in V_h$ satisfying 
\begin{equation}\label{eq:firstder}
a_h(\tilde{v},v) = (f, \sfp(v)) - \frac{1}{\alpha_k}(\sfp(v),w-\psi_h^{k-1}), \quad 
\forall v\in V_h.
\end{equation}
The coercivity of $a_h(\cdot,\cdot)$ and the boundedness of $\sfp$ imply the existence and uniqueness of $\tilde{v}\in V_h$ satisfying \eqref{eq:firstder}. In addition, it is easy to see that the solution map $W_h\ni w\mapsto \tilde{v}(w)\in V_h$ is bounded (uniformly in $h$).
Next, we define
\[
J(w) := \mathcal{L}(\tilde{v}(w),w), \quad \forall w\in W_h.
\]
Choosing $v=\tilde{v}(w)$ in \eqref{eq:firstder} and using the definition of $J$, we infer that
\[
J(w) = -\frac{\alpha_k}{2} a_h(\tilde{v}(w),\tilde{v}(w)) - (\mathcal{R}^\ast(w),1), \quad \forall w\in W_h.
\]
We want to show that there exists $w^\ast \in W_h$ such that
\begin{align}
    w^\ast = \argmax_{w \in W_h} J(w).
\label{eq:w_ast}
\end{align}
Since this is a finite-dimensional maximization problem, 
it suffices to show that $-J(w) \rightarrow \infty$ as $\|w\| \rightarrow \infty.$ 
To this end, we use the convexity of $\mathcal{R}^*$ (giving the subgradient inequality
$\mathcal{R}^\ast(w)\geq \mathcal{R}^\ast(0)-   \nabla \mathcal{R}^\ast (0)  w$
a.e.~in $\Omega$) and the coercivity property~\eqref{eq:coercivity_a} to obtain  
\begin{align*}
- J(w) \geq \frac{\alpha_k C_{\mathrm{coerc}}}{2} \|\tilde{v}(w)\|^2_{{V_h}} -(\nabla \mathcal{R}^\ast (0),  w) + (\mathcal{R}^\ast(0), 1).  
\end{align*}
To conclude, it suffices to show that there exists $c_1$ (possibly depending on $h$)
and $c_2$ so that $\|w\|\le c_1\|\tilde{v}(w)\|_{V_h} + c_2$. This, in turn, follows
from the fact that the inclusion $W_h \subseteq \sfp(V_h)$ implies the existence
of $\beta>0$ (possibly depending on $h$) so that, for all $w\in W_h$, there is $v'(w)\in V_h$
so that $\sfp(v'(w))=w$ and $\beta\|v'(w)\| \le \|w\|$. Then, 
testing \eqref{eq:firstder} with $v'(w)$ and invoking 
the continuity of $a_h$ (see \eqref{eq:continuity_a}) readily implies the expected bound. 
It follows that $w^\ast$ is well defined by \eqref{eq:w_ast}. 
Finally it remains to show that $(\tilde{v}(w^\ast),w^\ast)$ is a critical point of $\mathcal{L}$, but this is done using that $\mathcal{L}(v,\cdot)$ is concave, that the map $w\mapsto \tilde{v}(w)$ is continuous, and that $\mathcal{L}(\cdot,w)$ is continuous. Details are skipped as they are standard (see \cite[Exercise 7.16-4]{ciarlet2013linear}). 
\end{proof}

\begin{lemma}[Energy dissipation] \label[lemma]{lemma:energy}
Define the discrete energy 
\begin{equation}
    E_h(v) := \frac12  a_h(v,v) - (f, \sfp(v)), \quad \forall v \in V_h. 
\end{equation} 
The following property holds: For all $k \geq 1$, 
\begin{equation}
E_h(u_h^{k}) \leq E_h(u_h^{k-1}). 
\label{eq:energydis}
\end{equation}
\end{lemma}
\begin{proof}
The proof closely follows \cite[Lemma 3.2]{keith2025priori}. We provide it for completeness.
By symmetry of $a_h$ and by coercivity (see~\eqref{eq:coercivity_a}), we have  
\begin{align*}
E_h(u_h^{k}) - E_h(u_h^{k-1}) = &{}   a_h(u_h^{k}, u_h^{k} - u_h^{k-1}) - (f, \sfp(u_h^{k} - u_h^{k-1})) \\ & 
- \frac12 a_h(u_h^{k} - u_h^{k-1},  u_h^{k} - u_h^{k-1})  \\ 
\leq &  {} a_h(u_h^{k}, u_h^{k} - u_h^{k-1}) - (f, \sfp(u_h^{k} - u_h^{k-1})). 
\end{align*}
Thus, choosing $v_h = u_h^k-u_h^{k-1}$ in \eqref{eq:genPG1_bis}, subtracting \eqref{eq:genPG2_bis} for $k$ and $k-1$ and choosing $w_h = \psi_h^k-\psi_h^{k-1}$, we obtain 
\begin{align*}
E_h(u_h^{k}) - E_h(u_h^{k-1}) \leq  - \frac{1}{\alpha_{k}} (\nabla \mathcal{R}^*(\psi_h^{k}) -\nabla \mathcal{R}^*(\psi_h^{k-1}) , \psi_h^{k} - \psi_h^{k-1}) \leq 0,
\end{align*}
where the last inequality follows from the monotonicity property~\eqref{eq:monotonicity} of $\nabla \mathcal{R}^*$. 
\end{proof}

\subsection{Error analysis with minimal regularity} \label{sec:min_reg}

In what follows, we will invoke a bounded linear map $\mathcal{C}_h: L^2(\Omega)  \rightarrow \mathbb P^1(\mathcal{T}_h) \cap H^1_0(\Omega)$ that will be selected as a modified Clément-type interpolant in \Cref{sec:Ch}. This map will play the role of a smoother. In particular, the operator $\mathcal{C}_h$ allows us to define the consistency error $\delta_h(u) \in V_h'$ (where $V_h'$ is the dual space of $V_h$) 
as 
\begin{equation}
    \langle  \delta_h(u), v_h \rangle_{V_h',V_h} :=   a_h(\Pi_h u,v_h) - a(u, \mathcal{C}_h \sfp(v_h)), \quad \forall v_h \in V_h. 
    \label{eq:duality}
\end{equation} 
However, not any smoother operator is sufficient for our purposes as we need also to account for the feasible set $K$. In other words, we need a constraint-aware smoother $\mathcal{C}_h$. Such an operator is constructed in Section~\ref{sec:Ch}.

\begin{theorem}[Bound on discrete error] \label{thm:bestapprox}
For all $m \geq 1$, the following estimate holds: 
\begin{equation} \label{eq:main_estimate_discrete}
\begin{aligned}
&\frac{C_{\mathrm{coerc}}}{4} \|u_h^m - \Pi_hu\|^2_{V_h} + \frac{(\mathcal{D}(u, o_h^m),1)}{\sum_{k=1}^m \alpha_k} \\
& \leq \frac{ ( \mathcal{D}(u, o_h^0) ,1) }{\sum_{k=1}^m \alpha_k} + \frac{ 1}{C_{\mathrm{coerc}}} \|\delta_h(u)\|^2_{V_h'} + e_K (u_h^m) + e_h(u), 
\end{aligned}
\end{equation}
where $e_K$ and $e_h$ represent the feasibility and approximation errors, respectively defined as 
\begin{subequations} \begin{align}
e_K(u_h^m)  & := \inf_{v \in K} \Big\{ |a(u,  \mathcal{C}_h \sfp(u_h^m) -v) -  (f, \sfp(u_h^m) - v)| \Big\} \label{eq:e_K_feasibility},  \\ 
e_h(u) & := |a(u,\mathcal{C}_h J_h u-u) -  (f,J_h u-u) |.
\end{align} \end{subequations}
\end{theorem} 

\begin{proof}
Our starting point is the following identity: For all $v,w \in V_h$:
\begin{equation}
\label{eq:expanding_energy}
E_h(v)-E_h(w) = a_h(w,v-w)+\frac{1}{2} a_h(v-w,v-w) -(f, \sfp(v-w)).
\end{equation}
Taking $v := \Pi_h u$ and $w := u_h^k$ in \eqref{eq:expanding_energy} and recalling that $J_h:=\sfp\circ \Pi_h$, we obtain
\[
E_h(\Pi_h u) = E_h(u_h^k)
+a_h(u_h^k,\Pi_h u-u_h^k)
+\frac{1}{2} a_h(\Pi_h u - u_h^k, \Pi_h u - u_h^k) -(f, J_h u - \sfp(u_h^k)).
\]
Invoking the coercivity property~\eqref{eq:coercivity_a}, we infer that
\begin{multline}
\label{eq:energy1}
0  \geq E_h(u_h^k) - E_h(\Pi_h u) + a_h(u_h^k,\Pi_h u - u_h^k) \\  - (f, J_h u - \sfp(u_h^k)) + \frac{C_{\mathrm{coerc}}}{2} \|\Pi_h u - u_h^k\|_{V_{h}}^2.
\end{multline} 
Testing \eqref{eq:genPG1_bis} with $\Pi_h u - u_h^k$, we obtain 
\begin{align*}
a_h(u_h^k, \Pi_h u - u_h^k) - (f, J_h u - \sfp(u_h^k)) = -\frac{1}{\alpha_k} (J_h u - \sfp(u_h^k),
\psi_h^k-\psi_h^{k-1}).
\end{align*}
Thus, dropping the rightmost term in~\eqref{eq:energy1} which is nonnegative, we infer that
\begin{align*}
E_h(u_h^k) - E_h(\Pi_h u) 
+ \frac{1}{\alpha_k} (J_h u - \sfp(u_h^k),
\psi_h^{k-1}-\psi_h^{k}) \leq 0.
\end{align*}
Following \cite{keith2025priori}, we use the property~\eqref{eq:Pihprop} of $J_h$, the second step of the proximal DG iteration \eqref{eq:genPG2_bis}, the relation \eqref{eq:gradoh} between $\psi_h^k$ and $o_h^k$, and the three-point identity \eqref{eq:three_point_iden} to obtain 
\begin{align*}
 \frac{1}{\alpha_k} (J_h u - \sfp(u_h^k),
\psi_h^{k-1}-\psi_h^{k})
 & = \frac{1}{\alpha_k} (u - o_h^k, \psi_h^{k-1}-\psi_h^k)
\\ & =
 \frac{1}{\alpha_k}( u - o_h^k,  \nabla \mathcal{R}(o_h^{k-1}) - \nabla \mathcal{R}(o_h^{k})) \\ 
 & = \frac{1}{\alpha_k} ( \mathcal{D}(u, o_h^k) - \mathcal{D}(u,o_h^{k-1}) + \mathcal{D}(o_h^k,o_h^{k-1}),1) . \nonumber
\end{align*}
Combining the above expressions and using the fact that $ \mathcal{D}(o_h^k,o_h^{k-1}) \geq 0 $ yields 
\begin{align*}
E_h(u_h^k) -E_h(\Pi_h u) & + \frac{1}{\alpha_k} ( \mathcal{D}(u, o_h^k) - \mathcal{D}(u,o_h^{k-1}),1)   
 \leq 0.  
\end{align*}
We multiply the above equation by $\alpha_k>0$, sum from $k = 1$ to $k = m$, and use the energy dissipation property \eqref{eq:energydis} to obtain
\begin{align}\label{eq:inter2}
    \bigg( \sum_{k=1}^m\alpha_k\bigg) \Big( E_h(u_h^m) -E_h(\Pi_h u) \Big)  +  (\mathcal{D}(u, o_h^m),1)     \leq (\mathcal{D}(u, o_h^0),1) .
\end{align}
Thus, upon dividing \eqref{eq:inter2} by $\sum_{k=1}^m \alpha_k$, we arrive at 
\begin{align}
    E_h(u_h^m) -E_h(\Pi_h u)& 
    + \frac{ ( \mathcal{D}(u, o_h^m),1 )  }{\sum_{k=1}^m \alpha_k}  \leq \frac{( \mathcal{D}(u, o_h^0),1 ) }{\sum_{k=1}^m \alpha_k}. \label{eq:inter3}
\end{align}
Next, we find a lower bound for $E_h(u_h^m)-E_h(\Pi_h u)$. 
Choosing $w=\Pi_h u$ and $v=u_h^m$ in \eqref{eq:expanding_energy} gives
\begin{align} \label{eq:energy_expansion_after_summing}
    E_h(u_h^m) -E_h(\Pi_h u) = {}& a_h(\Pi_hu , u_h^m - \Pi_h u) - (f, \sfp(u_h^m) - J_h u) \\ &+ \frac{1}{2} a_h(u_h^m - \Pi_hu , u_h^m - \Pi_h u). \nonumber
\end{align}
We further split the first term using the consistency error defined in~\eqref{eq:duality} as
\begin{equation*}
    a_h(\Pi_hu , u_h^m - \Pi_h u) = \langle \delta_h (u),  u_h^m - \Pi_h u \rangle_{V_h', V_h} + a( u,  \mathcal{C}_h (\sfp(u_h^m) - J_h u)). 
\end{equation*}
Now, for any $v \in K$, we write $\mathcal{C}_h \sfp(u_h^m) - \mathcal{C}_h J_h u =(\mathcal{C}_h \sfp(u_h^m) - v)+(v -u) +( u -  \mathcal{C}_h J_h u)$. Using this expansion and the variational inequality \eqref{eq:vi}, we obtain 
\begin{align*}
 a_h(\Pi_hu , u_h^m - \Pi_h u) - (f, \sfp(u_h^m) - J_h u)  \geq  {} & \langle \delta_h (u),  u_h^m - \Pi_h u \rangle_{V_h', V_h}  \\  & + a(u, \mathcal{C}_h \sfp(u_h^m) - v) + a(u, u -\mathcal{C}_h J_h u) \\  & 
 - (f, \sfp(u_h^m)-v) - (f, u - J_hu). 
\end{align*}
Thus, owing to the coercivity property \eqref{eq:coercivity_a}, we infer that 
\begin{align*}
E_h(u_h^m) - E_h(\Pi_h u)    \geq{} & \frac{C_{\mathrm{coerc}}}{2} \|u_h^m - \Pi_h u\|_{V_h}^2 +\langle \delta_h (u),  u_h^m - \Pi_h u \rangle_{V_h', V_h}  \\ & + a(u, \mathcal{C}_h \sfp(u_h^m) - v) + a(u, u -\mathcal{C}_h J_h u) \\  & 
- (f,\sfp(u_h^m)-v) - (f, u - J_hu). 
\end{align*}
Substituting the above into \eqref{eq:inter3}, we obtain  
\begin{align}
\frac{C_{\mathrm{coerc}}}{2} \|u_h^m - \Pi_h u\|^2_{V_h} + \frac{ ( \mathcal{D}(u, o_h^m),1 ) }{\sum_{k=1}^m \alpha_k} \leq{} &  \frac{( \mathcal{D}(u, o_h^0),1 )  }{\sum_{k=1}^m \alpha_k} +  \langle \delta_h (u),  \Pi_h u - u_h^m \rangle_{V_h', V_h}\\ \nonumber & + a(u, v - \mathcal{C}_h \sfp(u_h^m) ) + a(u, \mathcal{C}_h J_h u- u) \\ \nonumber & 
- (f,v- \sfp( u_h^m) ) - (f, J_h u-u). 
\end{align}
Using the definition of $\|\delta_h(u)\|_{V_h'}$ and Young's inequality gives
\begin{align}
    \langle \delta_h (u),  \Pi_h u - u_h^m \rangle_{V_h',V_h} & \leq \|\delta_h(u)\|_{V_h'} \|\Pi_h u - u_h^m\|_{V_h} \\ 
& \leq \frac{C_{\mathrm{coerc}}}{4} \|u_h^m - \Pi_h u\|^2_{V_h} + \frac{ 1}{C_{\mathrm{coerc}}} \|\delta_h(u)\|^2_{V_h'}. \nonumber
\end{align}
Combining the above bounds and taking the infimum over $v\in K$ which is arbitrary yields the expected estimate. 
\end{proof}

\begin{corollary}[Best-approximation property] \label{cor:bestapprox}
For all $m \geq 1$, the following estimate holds: 
\begin{equation} \label{eq:main_estimate_general}
\begin{aligned}
&\frac{C_{\mathrm{coerc}}}{8C_{\sfp}^2} \|\sfp(u_h^m) - u\|^2_{\DG} + \frac{(\mathcal{D}(u, o_h^m),1)}{\sum_{k=1}^m \alpha_k} \\
& \leq \frac{ ( \mathcal{D}(u, o_h^0) ,1) }{\sum_{k=1}^m \alpha_k} + \frac{1}{C_{\mathrm{coerc}}} \|\delta_h(u)\|^2_{V_h'} + e_K (u_h^m) + e_h(u) + \frac{C_{\mathrm{coerc}}}{4C_{\sfp}^2}\|u - J_h u\|^2_{\DG}. 
\end{aligned}
\end{equation}
\end{corollary}

\begin{proof}
The estimate follows from~\eqref{eq:main_estimate_discrete}, the bound 
\[ \frac12\|\sfp(u_h^m)-u\|_{\DG}^2 \le \|u-J_hu\|_{\DG}^2+\|\sfp(u_h^m)-J_hu\|_{\DG}^2,
\] 
and the boundedness of $\sfp$ (see~\eqref{eq:bnd_sfp}) which gives
$\|\sfp(u_h^m)-J_hu\|_{\DG}^2 \le C_{\sfp}^2\|u_h^m-\Pi_hu\|_{V_h}^2$.
\end{proof}

\begin{theorem}[Convergence of $\lambda_h^m$] \label{thm:multiplier_convergence}
For all $m \geq 1$, the following holds:
\begin{align} \label{eq:multiplier_err_bound}
\|\lambda_h^m - \lambda\|_{H^{-1}(\Omega)}  \leq {} & C_\Pi C_{\mathrm{cont}} \|u_h^m - \Pi_h u\|_{V_h} + C_\Pi \|\delta_h(u)\|_{V_h'} \\ \nonumber 
& + \sup_{v\in H^1_0(\Omega), v \neq 0} \frac{|a(u, v - \mathcal{C}_h J_h v) - (f, v - J_h v) |}{\|\nabla v \|_{L
^2(\Omega)}}. 
\end{align}
\end{theorem}
\begin{proof}
Fix $v \in H^1_0(\Omega)$. Observe that by the first equation of the proximal DG iteration \eqref{eq:genPG1_bis} and the property~\eqref{eq:Pihprop} of $\Pi_h$, 
\begin{align*}
\langle \lambda_h^m, v\rangle = ( v, \lambda_h^m ) = (J_h v, \lambda_h^m) = a_h(u_h^m , \Pi_h v) - (f,J_h v).
\end{align*}
Furthermore, we use the definition~\eqref{eq:duality} of the consistency error $\delta_h(u)$ to write 
\begin{align*}
\langle \lambda, v \rangle 
& = a(u, v - \mathcal{C}_h J_h v)+ a(u,\mathcal{C}_h J_h v) - (f, v ) \\ 
& = a(u, v - \mathcal{C}_h J_h v) + a_h (\Pi_h u, \Pi_h v)  - \langle \delta_h(u) , \Pi_h v\rangle_{V_h' , V_h} - (f, v).
\end{align*}
Subtracting the above two identities gives 
\begin{align*}
\langle \lambda_h^m - \lambda, v\rangle 
= {}& a_h(u_h^m - \Pi_h u, \Pi_h v) +  \langle \delta_h(u) , \Pi_h v\rangle_{V_h' , V_h} \\ & + a(u, v - \mathcal{C}_hJ_h v) - (f, v - J_h v) \\  
\leq{}& ( C_{\mathrm{cont}} \|u_h^m - \Pi_h u\|_{V_h} + \|\delta_h(u)\|_{V_h'}) \|\Pi_h v\|_{V_h}  \\ & + a(u, v - \mathcal{C}_h J_h v) -(f, v - J_h v),
\end{align*}
where we used the boundedness of $a_h$ (see~\eqref{eq:continuity_a}).
To conclude, we invoke the stability of $\Pi_h$ from \eqref{eq:Pihprop} and the fact that 
$\|v\|_{\DG}=\|\nabla v\|_{L^2(\Omega)}$ since $v\in H^1_0(\Omega)$.
\end{proof} 

\begin{lemma}[Convergence of the bound-preserving approximation $o_h^m $] \label{lem:conv_ohm}
For all $m \geq 1$, the following holds:
\begin{align} \label{eq:bound_nonlinear_approx}
\|o_h^m - u \|_{L^2(\Omega)} \leq \| \sfp(u_h^m) - u \|_{L^2(\Omega)} + \|\Pi_h^0 u - u\| _{L^2(\Omega)},
\end{align}
where $\Pi_h^0$ is the $L^2$-projection onto $W_h$. (Recall that $W_h=\mathbb{P}^0(\mathcal{T}_h)$.)
\end{lemma} 
\begin{proof}
The second equation in the proximal DG iteration, see~\eqref{eq:genPG2_bis}, gives $o_h^m = \Pi_h^0 \sfp(u_h^m)$ . Hence, by the triangle inequality,
    \begin{align}
\|o_h^m - u \|_{L^2(\Omega)} \leq \|\Pi_h^0 (\sfp(u_h^m) - u) \|_{L^2(\Omega)} + \|\Pi_h^0 u - u\| _{L^2(\Omega)}, 
\end{align}
The $L^2$-stability of $\Pi_h^0$ leads to the expected bound.
\end{proof} 

\subsection{Error analysis with some additional regularity}

The error estimates derived so far do not require any additional regularity on the solution $u$. We now derive decay rates on the error under the assumption that $u \in H^2(\Omega) \cap H_0^1(\Omega)$ and $\phi \in H^2(\Omega)$ for simplicity. 
In what follows, for nonnegative real numbers $A$ and $B$, we use the abbreviated notation $A\lesssim B$ to denote the inequality $A\le CB$, where the value of $C$ can change at each occurrence provided it is independent of the mesh size $h$ and the proximity parameters $(\alpha_k)_{k\ge1}$. The value of $C$ can depend on the mesh shape-regularity, and on the higher-order norms of $u$ and $\phi$ according to the above assumptions.

Our analysis hinges on the following four assumptions. Assumptions~\ref{asm:Pih} and~\ref{asm:Ch} are standard assumptions on the stability and approximation properties of the operators $J_h:H^1(\mathcal{T}_h)\rightarrow V_{\mathcal{T}_h}$ and $\mathcal{C}_h:L^2(\Omega) \rightarrow \mathbb P^1(\mathcal{T}_h) \cap H^1_0(\Omega)$, respectively. Instead, Assumptions~\ref{asm:consistency} and~\ref{ass:Ch_K} are, respectively, related to the consistency error produced by the nonconforming discretization method and to how far the smoothed function $\mathcal{C}_h\sfp(u_h^m)$ lies from the feasible set $K$. Assumption~\ref{ass:Ch_K} is the key property that hinges on the constraint-awareness of the smoother $\mathcal{C}_h$. We will see how to realize this assumption in \Cref{sec:Ch}.
For now, we note that this requires an additional mild assumption (cf.~Lemma~\ref{lem:departure_K}).

\begin{assumption}[Stability and approximation for $J_h$]\label{asm:Pih}
The following estimates hold true:
\begin{subequations} \begin{alignat}{2}
&\|v - J_h v\|_{L^2(\Omega)} \lesssim h \|v\|_{\DG}, &\quad&\forall v \in H^1(\mathcal{T}_h), \label{eq:approx_H1_Jh} \\
&\|v - J_h v\|_{L^2(\Omega)} + h\|v - J_h v \|_{\DG} \lesssim h^2 |v|_{H^2(\Omega)}, &\quad&\forall v \in H^2(\Omega) \cap H_0^1(\Omega). \label{eq:approx_Jh}
\end{alignat}
\end{subequations}
\end{assumption}

\begin{assumption}[Stability and approximation for $\mathcal{C}_h$]\label{asm:Ch}
The following estimates hold true:
\begin{subequations} \begin{alignat}{2}
&\|\mathcal{C}_h v\|_{L^2(\Omega)} \lesssim \|v\|_{L^2(\Omega)}, &\quad&\forall v \in L^2(\Omega), \label{eq:stab_Ch} \\
&\|v - \mathcal{C}_h v\|_{L^2(\Omega)} \lesssim h \|v\|_{\DG}, &\quad&\forall v \in H^1(\mathcal{T}_h), \label{eq:approx_H1_Ch} \\
&\|v - \mathcal{C}_h v\|_{L^2(\Omega)} \lesssim h^2 |v|_{H^2(\Omega)}, &\quad&\forall v \in H^2(\Omega) \cap H_0^1(\Omega). \label{eq:approx_L2_Ch}
\end{alignat}
\end{subequations}
\end{assumption}

\begin{assumption}[Bound on consistency error]\label{asm:consistency}
The consistency error defined in~\eqref{eq:duality} satisfies
\begin{align}\label{eq:consist_gen}
\|\delta_h(u)\|_{V_h'} \lesssim h.
\end{align} 
\end{assumption}

\begin{assumption}[Departure from feasible set] \label{ass:Ch_K}
The following holds for all $m\geq 1$:
\begin{align} \label{eq:Ch_K}
\inf_{v\in K} \|\mathcal{C}_h\sfp(u_h^m) - v\|_{L^2(\Omega)} \lesssim h^2.
\end{align}
\end{assumption}

\begin{theorem}[Error estimates with additional regularity]\label{thm:general_error}
Under the above assumptions and assuming that the initialization of the proximal DG iterations satisfies $\|\psi_h^0\|_{L^\infty(\Omega)} \lesssim 1$, the following holds for all $m \geq 1$:
\begin{align}\label{eq:general_error}
\|\sfp(u_h^m) - u\|^2_{\DG} + \|\lambda_h^m - \lambda\|^2_{H^{-1}(\Omega)} + \|o_h^m-u\|_{L^2(\Omega)}^2  \lesssim \frac{1}{\sum_{k=1}^m \alpha_k} + h^2.
\end{align}
\end{theorem}

\begin{proof}
We first bound $\|\sfp(u_h^m) - u\|^2_{\DG}$. To this purpose, we invoke Corollary~\ref{cor:bestapprox} observing that both terms on the left-hand side of~\eqref{eq:main_estimate_general} are nonnegative, so that it suffices to estimate the five terms on the right-hand side, say $T_1,\ldots,T_5$.
Our assumption on $\psi_h^0$ readily gives $(\mathcal{D}(u, o_h^0), 1) \lesssim 1$ so that
\[
T_1 \lesssim \frac{1}{\sum_{k=1}^m \alpha_k}.
\]
By \Cref{asm:consistency} on the consistency error, we have $T_2 \lesssim h^2$.
To bound $T_3 = e_K(u_h^m)$, we use that $\mathcal{C}_h\sfp(u_h^m)-v \in H^1_0(\Omega)$ and the
regularity of $u$ to integrate by parts. This gives, for all $v\in K$,
\begin{align*}
&{} |a(u,  \mathcal{C}_h \sfp(u_h^m) -v) -  (f, \sfp(u_h^m) - v)|\\ & \le{} 
|(\Delta u + f,v-\mathcal{C}_h \sfp(u_h^m))| + |(f,\mathcal{C}_h \sfp(u_h^m) -\sfp(u_h^m))| \\
& \lesssim \|v-\mathcal{C}_h \sfp(u_h^m)\|_{L^2(\Omega)} + \|\mathcal{C}_h \sfp(u_h^m) -\sfp(u_h^m)\|_{L^2(\Omega)}.
\end{align*}
Hence,
\[
T_3 \lesssim \inf_{v\in K} \|v-\mathcal{C}_h \sfp(u_h^m)\|_{L^2(\Omega)} 
+ \|\mathcal{C}_h \sfp(u_h^m) -\sfp(u_h^m)\|_{L^2(\Omega)}.
\]
Owing to Assumption~\ref{ass:Ch_K}, the first term on the right-hand side is bounded by $h^2$. Concerning the second term, invoking the triangle inequality, the approximation property~\eqref{eq:approx_H1_Ch} for $\mathcal{C}_h$, the $L^2$-stability of $\mathcal{C}_h$ from~\eqref{eq:stab_Ch}, and the approximation properties~\eqref{eq:approx_Jh} and~\eqref{eq:approx_L2_Ch} for $J_h$ and $\mathcal{C}_h$, respectively, we infer that
\begin{align*}
\|\mathcal{C}_h \sfp(u_h^m) -\sfp(u_h^m)\|_{L^2(\Omega)} \le{}& 
\|(I-\mathcal{C}_h)(\sfp(u_h^m)-J_hu)\|_{L^2(\Omega)} + \|\mathcal{C}_h(u-J_hu)\|_{L^2(\Omega)} \\
&+ \|u-\mathcal{C}_hu\|_{L^2(\Omega)} + \|u-J_hu\|_{L^2(\Omega)} \\
\lesssim {}& h\|\sfp(u_h^m)-J_hu\|_{\DG} + h^2.
\end{align*}
Invoking the triangle inequality $\|\sfp(u_h^m)-J_hu\|_{\DG} \le \|\sfp(u_h^m)-u\|_{\DG} + \|u-J_hu\|_{\DG}$, Young's inequality, and \eqref{eq:approx_Jh} to bound $\|u-J_hu\|_{\DG}$, we infer that
\[
\|\mathcal{C}_h \sfp(u_h^m) -\sfp(u_h^m)\|_{L^2(\Omega)} + \gamma 
\|\sfp(u_h^m)-u\|_{\DG}^2 \lesssim h^2,
\]
where $\gamma>0$ can be chosen as small as needed. Altogether, this gives
\[
T_3 + \gamma \|\sfp(u_h^m)-u\|_{\DG}^2 \lesssim h^2.
\]
Invoking similar arguments, we infer that
\begin{align*}
|a(u,\mathcal{C}_h J_h u-u) -  (f,J_h u-u)| \lesssim {}& \|\mathcal{C}_h J_h u-u\|_{L^2(\Omega)}
+ \|J_h u-u\|_{L^2(\Omega)}  \\
\le{}& \|\mathcal{C}_h (J_h u-u)\|_{L^2(\Omega)} + \|\mathcal{C}_hu-u\|_{L^2(\Omega)} \\ & 
+ \|J_h u-u\|_{L^2(\Omega)} \\
\lesssim{}& \|\mathcal{C}_hu-u\|_{L^2(\Omega)}
+ \|J_h u-u\|_{L^2(\Omega)} \lesssim h^2.
\end{align*}
Hence, $T_4\lesssim h^2$. Finally, $T_5\lesssim h^2$ readily follows from~\eqref{eq:approx_Jh}. Putting the above five bounds together and taking $\gamma>0$ small enough gives the expected bound on $\|\sfp(u_h^m) - u\|^2_{\DG}$.

To bound $\|\lambda_h^m - \lambda \|_{H^{-1}(\Omega)}^2$, we invoke \Cref{thm:multiplier_convergence} and estimate the three terms on the right-hand side of~\eqref{eq:multiplier_err_bound}, say $T'_1,T'_2,T'_3$. To estimate the first term, we use Theorem~\ref{thm:bestapprox}, observing that the four terms on the right-hand side of~\eqref{eq:main_estimate_discrete} have been estimated in the first step of the present proof. This gives $T'_1\lesssim h^2$. By \Cref{asm:consistency} on the consistency error, we have $T'_2 \lesssim h^2$. Finally, for $T'_3$, we fix $v \in H^1_0(\Omega)$ and observe that
\begin{align*}
a(u, v - \mathcal{C}_hJ_hv ) - (f, v - J_h v)) \lesssim{}& 
\|v - \mathcal{C}_hJ_hv  \|_{L^2(\Omega)}  +  \|v - J_h v\|_{L^2(\Omega)} \\
\lesssim{}& \|v-\mathcal{C}_hv\|_{L^2(\Omega)} + \|v - J_h v\|_{L^2(\Omega)} \lesssim h\|\nabla v\|_{L^2(\Omega)},
\end{align*}
where we used the triangle inequality, the $L^2$-stability of $\mathcal{C}_h$ from~\eqref{eq:stab_Ch}, the approximation properties~\eqref{eq:approx_H1_Ch} and~\eqref{eq:approx_H1_Jh} for $\mathcal{C}_h$ and $J_h$, respectively, and finally the fact that $\|v\|_{\DG} = \| \nabla v \|_{L^2(\Omega)}$ since $v \in H^1_0(\Omega)$. Combining the three above bounds leads to the expected bound on $\|\lambda_h^m - \lambda \|_{H^{-1}(\Omega)}^2$. 

Finally, the bound on $\|o_h^m-u\|_{L^2(\Omega)}^2$ follows from Lemma~\ref{lem:conv_ohm}, the broken Poincar\'e inequality (stating that $\|v\|_{L^2(\Omega)}\lesssim \|v\|_{\DG}$ for all $v\in H^1(\mathcal{T}_h)$, see \cite{brenner2003poincare}, and applied to $v:=\sfp(u_h^m)-u$), the above bound on $\|\sfp(u_h^m)-u\|_{\DG}$, and the fact that $\|\Pi_h^0u-u\|_{L^2(\Omega)}\lesssim h^2$.
\end{proof}

\subsection{The operator \texorpdfstring{$\mathcal{C}_h$}{Ch}} \label{sec:Ch}

As motivated above, the operator $\mathcal{C}_h:L^2(\Omega) \rightarrow \mathbb P^1(\mathcal{T}_h) \cap H^1_0(\Omega)$ plays the role of a smoother, but at the same time should control the departure of $u_h^m$ from the feasible set $K$, see \Cref{ass:Ch_K}.
To motivate the devising of $\mathcal{C}_h$, we identify a property satisfied by $u_h^m$ for all the proximal DG methods considered herein. This property hinges on the fact that we always have $W_h=\mathbb{P}^0(\mathcal{T}_h)$.

\begin{lemma}[Property of proximal iterates] \label[lemma]{lemma:Kh}
For all $m \geq 1$, we have
\begin{equation} \label{eq:discreteK}
\sfp(u_h^m) \in K_h :=\left \{ v \in L^2(\Omega)  \mid  \int_{T} (v - \phi) \mathrm{d}x \geq 0 \;\; \forall T \in \mathcal{T}_h \right \}. 
\end{equation} 
\end{lemma}

\begin{proof}
The claim follows by testing~\eqref{eq:genPG2_bis} with the indicator function associated with one fixed element $T\in\mathcal{T}_h$ (this is possible since $\mathbb{P}^0(\mathcal{T}_h)\subseteq W_h$) and using that $\nabla \mathcal{R}^*(\psi_h^m) > \phi$ a.e. in $T$.
\end{proof}

We denote by $\mathcal{N}_h$ (resp.\ $\mathcal{N}_h^{\partial}$) the set of all interior (resp.\ boundary) vertices of $\mathcal{T}_h$.
Following \cite{Fuhrer+2024+363+378}, we consider the operator $\mathcal{C}_h: L^2(\Omega) \rightarrow \mathbb{P}^{1}(\mathcal{T}_h) \cap H_0^1(\Omega)$ defined as
\begin{equation} \label{eq:clement_case_1}
\mathcal{C}_h v := \sum_{z \in \mathcal{N}_h} v_z  \varphi_z,  \qquad v_z := \sum_{T \subset \omega_z } \frac{\alpha_{z,T}}{|T|}\int_{T} v \dd x, 
\end{equation}
where $(\varphi_z)_{z\in\mathcal{N}_h}$ are the nodal Lagrange basis functions of $\mathbb{P}^{1}(\mathcal{T}_h) \cap H_0^1(\Omega)$, and the weights $\alpha_{z,T}$ are chosen such that
\begin{equation}
z = \sum_{T \subset \omega_z} \alpha_{z,T} s_T, \quad  \sum_{T \subset \omega_z} \alpha_{z,T} = 1,  \quad \alpha_{z,T} \geq 0, \label{eq:convex_hull_z}
\end{equation}
where $s_T$ is the centroid of $T \in \mathcal{T}_h$ and $\omega_z$ is the union of elements sharing the node $z$. The choice of $\{\alpha_{z,T}\}$ is not unique, 
but is always possible since every interior node $z \in \mathcal{N}_h$ belongs to the convex hull of the set $\{ s_T \mid T \subset \omega_z\}$. 

\begin{lemma}[Stability and approximation for $\mathcal{C}_h$] \label[lemma]{lemma:modified_Clement}
Assumption~\ref{asm:Ch} holds true.
\end{lemma}

\begin{proof}
The properties~\eqref{eq:stab_Ch} and~\eqref{eq:approx_L2_Ch} are already established in \cite[Theorem 11]{Fuhrer+2024+363+378}. 
It remains to prove~\eqref{eq:approx_H1_Ch}. Let $v \in H^1(\mathcal{T}_h)$ and let us prove that $\|v - \mathcal{C}_h v\|_{L^2(\Omega)} \lesssim h \|v\|_{\DG}$. 
We consider an $L^2$-stable interpolation operator considered in~\cite{ern2017finite}. Specifically, let $\mathsf{S}_h:H^1(\mathcal{T}_h)\rightarrow \mathbb P^1(\mathcal{T}_h) \cap H^1_0(\Omega)$ be defined as $\mathsf{S}_h:=\mathsf{E}_h \circ \Pi_h^1$, where $\Pi_h^1$ is the $L^2$-orthogonal projection onto $\mathbb{P}^1(\mathcal{T}_h)$ and $\mathsf{E}_h: \mathbb{P}^1(\mathcal{T}_h) \rightarrow \mathbb{P}^1(\mathcal{T}_h) \cap H_0^1(\Omega)$ is the averaging/Oswald enriching map originally considered, among others, in \cite{Brenner:93,Oswald:93}. It is readily shown that, for all $v\in H^1(\mathcal{T}_h)$,
\[
\|v - \mathsf{S}_h v\|_{L^2(\Omega)} + h \| \nabla \mathsf{S}_h v \|_{L^2(\Omega)}  \lesssim h \|v\|_{\DG}.  
\]
We then use the triangle inequality and the $L^2$-stability of $\mathcal{C}_h$ to infer that
\begin{align*}
\|v - \mathcal{C}_h v \|_{L^2(\Omega)} & \leq \|v - \mathsf{S}_h v\|_{L^2(\Omega)} + \|\mathcal{C}_h(v - \mathsf{S}_h v)\|_{L^2(\Omega)}+ \|\mathsf{S}_h v - \mathcal{C}_h \mathsf{S}_h v\|_{L^2(\Omega)}  \\  
& \lesssim \|v - \mathsf{S}_h v\|_{L^2(\Omega)} + h \| \nabla \mathsf{S}_h v \|_{L^2(\Omega)}
\lesssim h \|v\|_{\DG},
\end{align*}
where we also used the approximation estimate $\|w-\mathcal{C}_hw\|_{L^2(\Omega)} \lesssim h\|\nabla w\|_{L^2(\Omega)}$ for all $w\in H^1_0(\Omega)$, which is already established in~\cite[Theorem 11]{Fuhrer+2024+363+378}. This completes the proof. 
\end{proof}

We are now ready to examine the departure from the feasible set of the proximal primal iterates $u_h^m$ using the operator $\mathcal{C}_h$. For this purpose, we make a mild assumption on the obstacle function $\phi$. Recall that we are already assuming that $\phi\in H^2(\Omega)$ and that $\phi|_{\partial\Omega}\leq0$.

\begin{lemma}[Departure from feasible set] \label{lem:departure_K}
Assume that $\phi \vert_{\partial \Omega} \in \mathbb{P}^1(\mathcal{F}^\partial_h)$. Then
Assumption~\ref{ass:Ch_K} holds true.
\end{lemma}

\begin{proof} 
Following \cite{keith2025priori}, 
let us set
\[
v(u_h^m) := \mathcal{C}_h \sfp(u_h^m) -\tilde{\mathcal{C}}_h \phi + \phi,
\]
where the modified interpolant $\tilde{\mathcal{C}}_h: H^1(\Omega) \rightarrow \mathbb P^1(\mathcal{T}_h) \cap H^1(\Omega)$ is a modification of the constraint-aware smoother $\mathcal{C}_h$ that accounts for non-homogeneous boundary conditions. Namely, for every interior vertex $z \in \mathcal{N}_h$, we set $\tilde{\mathcal{C}}_h v(z) = \mathcal{C}_h v(z)$, and for every boundary vertex $z \in  \mathcal{N}_h^\partial$, we set $\tilde{\mathcal{C}}_h v(z)  = (\mathcal{SZ}_h v )(z)$, where $\mathcal {SZ}_h: H^1(\Omega) \rightarrow \mathbb{P}^1(\mathcal{T}_h) \cap H^1(\Omega)$ is the Scott--Zhang interpolant defined in \cite{scott1990finite}. It is shown in \cite[Lemma 4.2]{keith2025priori} that $\|\phi-\tilde{\mathcal{C}}_h \phi\|_{L^2(\Omega)} \lesssim h^2$.

Let us now verify that $v(u_h^m)\in K$. For every boundary vertex $z \in \mathcal{N}_h^\partial$, we observe that $\mathcal{C}_h \sfp(u_h^m)(z)=0$ by definition of $\mathcal{C}_h$ and $\tilde{\mathcal{C}}_h \phi(z) = \phi(z)$ since $\phi \vert_{\partial \Omega} \in \mathbb{P}^1(\mathcal{F}^\partial_h)$. Thus, $v(u_h^m)(z)=0$ for all $z \in \mathcal{N}_h^\partial$, and since $v(u_h^m)\vert_{\partial \Omega} \in \mathbb{P}^1(\mathcal{F}^\partial_h)$, we conclude that $v(u_h^m)\in H^1_0(\Omega)$. Moreover, we have $(\mathcal{C}_{h} \sfp(u_h^m) - \tilde{\mathcal{C}}_h \phi) (z) \ge 0$ for every boundary vertex $z\in\mathcal{N}_h^\partial$. 
Consider next an interior vertex $z \in \mathcal{N}_h$. Owing to Lemma~\ref{lemma:Kh} and the definition of $\mathcal{C}_h$, we readily see that 
\begin{align*}
(\mathcal{C}_{h} \sfp(u_h^m) - \tilde{\mathcal{C}}_h \phi) (z)  = \sum_{T \subset \omega_z } \frac{\alpha_{z,T}}{|T|}\int_{T} (\sfp(u_h^m) - \phi) \dd x \geq 0, \qquad \forall z\in\mathcal{N}_h. 
\end{align*}
Altogether, we see that $(\mathcal{C}_{h} \sfp(u_h^m) - \tilde{\mathcal{C}}_h \phi) (z) \ge 0$ for every $z\in\mathcal{N}_h \cup \mathcal{N}_h^\partial$. 
Therefore, the continuous piecewise affine function $\mathcal{C}_h \sfp(u_h^m) -\tilde{\mathcal{C}}_h \phi$ takes nonnegative values everywhere in $\Omega$. We conclude that $v(u_h^m)\ge \phi$ in $\Omega$, i.e., $v(u_h^m)\in K$. 

We are now ready to establish the estimate~\eqref{eq:Ch_K} in Assumption~\ref{ass:Ch_K}.
We observe that
\[
\inf_{v\in K} \|\mathcal{C}_h\sfp(u_h^m) - v\|_{L^2(\Omega)} \le \|\mathcal{C}_h\sfp(u_h^m) - v(u_h^m)\|_{L^2(\Omega)} = \|\phi-\tilde{\mathcal{C}}_h \phi\|_{L^2(\Omega)} \lesssim h^2.
\]
This concludes the proof.
\end{proof}

\section{Specific proximal DG methods}
\label{sec:specific_methods}
We now apply the general framework of \Cref{sec:error_analysis} for each of the proximal DG methods introduced in \Cref{sec:dG_intro,sec:hybrid_intro}. For each method, 
we first specify the operators $\sfp$ and $\Pi_h$, and the discrete bilinear form $a_h$.
Then we verify that Assumption~\ref{ass:basic} (basic setting), Assumption~\ref{asm:Pih} (stability and approximation for $J_h$), and Assumption~\ref{asm:consistency} (bound on consistency error) hold true. Recall that Assumptions~\ref{asm:Ch} and~\ref{ass:Ch_K} were already established in \Cref{sec:Ch}. At the end of the section, when dealing with the proximal HHO method, we introduce a different approach to bound the error inspired from \cite{Cicuttin2020} and which can be invoked whenever $V_{\mathcal{T}_h}=\mathbb{P}^0(\mathcal{T}_h)$. This turns out to be the case only for some of the HHO methods, but not for the other nonconforming discretization methods.

\subsection{Proximal IPDG}\label{sec:IPDG}
Recall that the discrete spaces are $V_h := \mathbb{P}^1(\mathcal{T}_h)$ and $W_h := \mathbb{P}^0(\mathcal{T}_h)$ (see \Cref{sec:dG_intro}). We equip the space $V_h$ with the $\|{\cdot}\|_{\DG}$-norm. 
We employ the symmetric IPDG bilinear form $a_h$ defined for all $z_h, v_h \in V_h$ as
\begin{align}
\label{eq:ipdg_form}
a_h(z_h,v_h) = & \sum_{T \in \mathcal{T}_h}(\nabla z_h, \nabla v_h)_T
-\sum_{F\in\mathcal{F}_h \cup \mathcal{F}_h^\partial } ( \{\nabla z_h\}\cdot \bfn_F, [v_h] )_{F}\\
& -\sum_{F\in\mathcal{F}_h \cup \mathcal{F}_h^\partial } ( \{\nabla v_h\}\cdot \bfn_F, [z_h] )_{F}
+\sum_{F\in\mathcal{F}_h \cup \mathcal{F}_h^\partial }  \frac{\sigma}{h_F} ( [z_h], [v_h] )_F, \nonumber
\end{align}
where $\sigma>0$ is a user-specified penalty parameter, assumed to be large enough so that $a_h$ is coercive with respect to the $\|{\cdot}\|_{\DG}$-norm, see \cite{riviere2008discontinuous} or \cite[Section 4.1]{di2011mathematical}. In particular, with standard arguments, one can show that \eqref{eq:coercivity_a} and \eqref{eq:continuity_a} hold true. 
Here, $V_{\mathcal{T}_h} = V_h$, $\sfp$ is the identity map, and $\Pi_h$ is the $L^2$-orthogonal projection onto $V_h$. Hence, \eqref{eq:bnd_sfp} trivially holds true, and \eqref{eq:Pihprop} is a classical property of the $L^2$-orthogonal projection onto $V_h$. Altogether, this shows that \Cref{ass:basic} is satisfied. Moreover, \Cref{asm:Pih} also holds true since $J_h:=\sfp\circ \Pi_h$ is again the $L^2$-orthogonal projection onto $V_h$. It remains to verify \Cref{asm:consistency}.

\begin{lemma}[Bound on consistency error]\label[lemma]{lemma:consistency_IPDG}
Assumption~\ref{asm:consistency} holds true.
\end{lemma}

\begin{proof}
For all $v_h \in V_h$, we can write (since $\mathcal{C}_h v_h \in H_0^1(\Omega)$)
$\langle  \delta_h(u), v_h \rangle_{V_h', V_h} = \delta_1-\delta_2$ with
\begin{align*} 
\delta_1 &:= a_h(\Pi_hu,v_h) + (\Delta u,v_h), \\
\delta_2 &:= (\Delta u,v_h-\mathcal{C}_hv_h).
\end{align*}
It is well-known, see, e.g., \cite[Chapter 2]{riviere2008discontinuous} or \cite[Section 4.1]{di2011mathematical}, that
\begin{align*}
|\delta_1| \lesssim{}& \Bigg( \|\nabla(\Pi_hu-u)\|_{L^2(\Omega)}^2 + \sum_{F \in \mathcal{F}_h \cup \mathcal{F}_h^\partial} h_F \|\{ \nabla (\Pi_h u - u)\} \cdot \bm n_F \|_{L^2(F)}^2 \\
& + \sum_{F \in \mathcal{F}_h \cup \mathcal{F}_h^\partial}  h_F^{-1} \| [\Pi_hu] \|_{L^2(F)}^2 \Bigg)^{\frac12} \|v_h\|_{\DG}  \lesssim h \|v_h\|_{\DG}, 
\end{align*} 
where we used the approximation properties of $\Pi_h$.
In addition, invoking~\eqref{eq:approx_H1_Ch} gives $|\delta_2| \lesssim h \|v_h\|_{\DG}$. 
Combining the above two estimates completes the proof. 
\end{proof} 

Everything is now in place to infer the error estimate for the proximal IPDG method from \Cref{thm:general_error}.

\begin{corollary}[Error estimate for proximal IPDG]\label{cor:errorDG}
Assume that $u \in H^2(\Omega) \cap H^1_0(\Omega)$ and $\phi \in H^2(\Omega)$ with $\phi \vert_{\partial \Omega} \in \mathbb{P}^1(\mathcal{F}^\partial_h)$. Then, the proximal IPDG iterates generated  by \Cref{alg:main_alg_discrete} satisfy the following error estimate: For all $m \geq 1$, 
\begin{align}
   \|u_h^m - u\|^2_{\DG}  + \|\lambda_h^m - \lambda \|_{H^{-1}(\Omega)}^2 + \|o_h^m - u\|_{L^2(\Omega)}^2 & \lesssim \frac{1}{\sum_{k=1}^m \alpha_k} + h^2. 
\end{align}
\end{corollary}

\subsection{Proximal EG}
\label{sec:EG}
The analysis of the proximal EG method closely follows that of the proximal IPDG method. The only difference lies in the choice of discrete spaces, which are now taken to be
$V_h := (\mathbb{P}^1(\mathcal{T}_h)\cap H_0^1(\Omega)) + \mathbb{P}^0(\mathcal{T}_h)$
and $W_h := \mathbb{P}^0(\mathcal{T}_h).$
The space $V_h$ is again equipped with the DG norm defined in \eqref{eq:DGnorm}.
The EG bilinear form $a_h$ is identical to the IPDG bilinear form defined by \eqref{eq:ipdg_form}, and the coercivity of $a_h$ again 
holds true if the penalty parameter $\sigma$ is large enough. 
It is then straightforward to verify that the error estimate from 
Corollary~\ref{cor:errorDG} holds true for the proximal EG method as well. 

\subsection{Proximal H-IP}\label{sec:HIP}

The H-IP method introduces facet unknowns to enable static condensation; cf.~\Cref{rem:StaticCondensation}. Recall from \Cref{sec:hybrid_intro} that the discrete spaces are defined as
$V_h := V_{\mathcal{T}_h} \times V_{\mathcal{F}_h} := \mathbb{P}^1(\mathcal{T}_h) \times \mathbb{P}^1(\mathcal{F}_h)$ and $W_h := \mathbb{P}^0(\mathcal{T}_h),$ 
and that a generic element of $V_h$ is denoted by $v_h = (v_{\mathcal{T}_h}, v_{\mathcal{F}_h}) \in V_h$, while its restriction to $T \in \mathcal{T}_h$ is denoted by $(v_T, v_{\partial T}) := (v_{\mathcal{T}_h}|_T,\, v_{\mathcal{F}_h}|_{\partial T}) \in V_T := \mathbb{P}^1(T) \times \mathbb{P}^1(\mathcal{F}_T)$.
The space $V_h$ is equipped with the norm
\begin{align}\label{eq:HIPnorm}
\|v_h\|_{V_h}^2 := \sum_{T \in \mathcal{T}_h} \bigg\{ \|\nabla v_T\|_{L^2(T)}^2
+ h_T^{-1} \|v_T - v_{\partial T}\|_{L^2(\partial T)}^2 \bigg\}, \qquad \forall v_h \in V_h.
\end{align}
The local H-IP bilinear form $a_T$ is defined on $V_T \times V_T$ as
\begin{align}
a_T((z_T,z_{\partial T}),(v_T,v_{\partial T})) := {}& (\nabla z_T, \nabla v_T)_T
- ( \nabla z_T \cdot \bfn_T, v_T - v_{\partial T} )_{\partial T} \nonumber \\ 
& - ( \nabla v_T \cdot \bfn_T, z_T - z_{\partial T} )_{\partial T}
+ \frac{\sigma}{h_T} ( z_T - z_{\partial T}, v_T - v_{\partial T})_{\partial T},
\end{align}
and the global bilinear form $a_h$ on $V_h \times V_h$ is defined as
\begin{equation}
\label{eq:hip_form}
a_h(z_h, v_h) := \sum_{T \in \mathcal{T}_h} a_T((z_T,z_{\partial T}),(v_T,v_{\partial T})).
\end{equation}
If the penalty parameter $\sigma$ is large enough, the discrete bilinear form $a_h$ is coercive (see \cite[Theorem~3.3]{FabienKnepleyRiviere2019}). Moreover, the discrete bilinear form $a_h$ is also continuous. Hence, properties \eqref{eq:coercivity_a} and \eqref{eq:continuity_a} hold true. 

The linear map $\sfp:V_h\to \mathbb{P}^1(\mathcal{T}_h)$ is such that $\sfp(v_h)=v_{\mathcal{T}_h}$ for all $v_h = (v_{\mathcal{T}_h}, v_{\mathcal{F}_h}) \in V_h$. Moreover, the linear map 
$\Pi_h: H^1(\mathcal{T}_h) \rightarrow V_h$ is defined as
\begin{equation} 
\label{eq:Pi_h_hybrid}
\Pi_h v := (\Pi^1_{\mathcal{T}_h} v, \Pi^1_{\mathcal{F}_h} \{ v \} ), \qquad \forall v\in H^1(\mathcal{T}_h), 
\end{equation}
where $\Pi^1_{\mathcal{T}_h}$ and $\Pi^1_{\mathcal{F}_h}$ are the $L^2$-orthogonal projections onto $V_{\mathcal{T}_h}=\mathbb{P}^1(\mathcal{T}_h)$ and $V_{\mathcal{F}_h}=\mathbb{P}^1(\mathcal{F}_h)$, respectively. Notice that both operators are defined locally: $(\Pi^1_{\mathcal{T}_h} v)|_T := \Pi^1_T(v|_T)$ for all $T \in \mathcal{T}_h$, with $\Pi^1_T$ the $L^2$-orthogonal projection onto $\mathbb{P}^1(T)$, and $(\Pi^1_{\mathcal{F}_h} \mu)|_F := \Pi^1_F(\mu|_F)$ for all $F \in \mathcal{F}_h \cup \mathcal{F}_h^\partial$, with $\Pi^1_F$ the $L^2$-orthogonal projection onto $\mathbb{P}^1(F)$. 
The following result shows that Assumption~\ref{ass:basic} indeed holds true.

\begin{lemma}[Fulfillment of Assumption~\ref{ass:basic}] \label{lem:H-IP_basic}
Properties~\eqref{eq:bnd_sfp} and~\eqref{eq:Pihprop} hold true.
\end{lemma}

\begin{proof}
Proof of~\eqref{eq:bnd_sfp}. We need to verify that $\|\sfp(v_h)\|_{\DG} \lesssim \|v_h\|_{V_h}$ for all $v_h\in V_h$. This is a classical result that uses the shape-regularity of the mesh and the fact that $v_{\mathcal{F}_h}|_F = 0$ for all $F \in \mathcal{F}_h^{\partial}$, see, e.g., \cite[Equ.~(2.20)]{KirkRiviereMasri2023}.

Proof of~\eqref{eq:Pihprop}. We first show that $\|\Pi_h v\|_{V_h} \lesssim \|v\|_{\DG}$ for all $v\in H^1(\mathcal{T}_h)$. We observe that 
\[
\|\Pi_h v\|_{V_h}^2 =
\sum_{T\in\mathcal{T}_h}\|\nabla \Pi^1_T v\|_{L^2(T)}^2
+ \sum_{T\in\mathcal{T}_h} h_T^{-1}\|\Pi^1_T v - \Pi^1_{\mathcal{F}_h}\{v\}\|_{L^2(\partial T)}^2. 
\]
For the first term on the right-hand side, the $H^1$-stability of $\Pi^1_T$ readily gives that
$\|\nabla \Pi^1_T v \|_{L^2(T)}$ $\lesssim \|\nabla v \|_{L^2(T)}$ for all $T\in\mathcal{T}_h$.
To bound the second term, we consider a face $F\in\mathcal{F}_T$. If $F$ is a boundary face,
we have
\begin{align*}
\|\Pi^1_T v - \Pi^1_{\mathcal{F}_h}\{v\}\|_{L^2(F)} &= \|\Pi^1_T v - \Pi^1_Fv\|_{L^2(F)}
= \|\Pi^1_F(\Pi^1_Tv-v)\|_{L^2(F)} \\
&\le \|\Pi^1_Tv-v\|_{L^2(F)} \lesssim h_T^{\frac12} \|\nabla v\|_{L^2(T)}.
\end{align*}
Instead, if $F$ is an interior face, we obtain using similar arguments (with $\epsilon_{T,F}=\bfn_T{\cdot}\bfn_F=\pm1$)
\begin{align*}
\|\Pi^1_T v - \Pi^1_{\mathcal{F}_h}\{v\}\|_{L^2(F)} &= \|\Pi^1_T v - \Pi^1_F(v|_T + \epsilon_{T,F}\tfrac12 [v])\|_{L^2(F)} \\
&\le \|\Pi^1_F(\Pi^1_Tv-v)\|_{L^2(F)} + \tfrac12 \|\Pi^1_F[v]\|_{L^2(F)}\\
&\lesssim h_T^{\frac12} \|\nabla v\|_{L^2(T)} + \|[v]\|_{L^2(F)}.
\end{align*}
Collecting the above bounds, squaring, and summing over all the elements $T\in\mathcal{T}_h$ proves that $\|\Pi_h v\|_{V_h} \lesssim \|v\|_{\DG}$ for all $v\in H^1(\mathcal{T}_h)$. 
Finally, since $J_h:\sfp \circ \Pi_h = \Pi^1_{\mathcal{T}_h}$ and $W_h=\mathbb{P}^0(\mathcal{T}_h) \subset V_{\mathcal{T}_h}:= \mathbb{P}^1(\mathcal{T}_h)$, we readily infer that 
$(J_hv-v,w_h)=0$ for all $v\in H^1(\mathcal{T}_h)$ and all $w_h\in W_h$.
\end{proof}

Since $J_h$ is the $L^2$-orthogonal projection onto $V_{\mathcal{T}_h}:= \mathbb{P}^1(\mathcal{T}_h)$, Assumption~\ref{asm:Pih} is satisfied. It remains to verify Assumption~\ref{asm:consistency}.

\begin{lemma}[Bound on consistency error]\label{lemma:consistency_HIP}
Assumption~\ref{asm:consistency} holds true.
\end{lemma}

\begin{proof}
For all $v_h \in V_h$, we can write (since $\mathcal{C}_h \sfp(v_h) \in H_0^1(\Omega)$)
$\langle  \delta_h(u), v_h \rangle_{V_h', V_h} = \delta_1-\delta_2$ with
\begin{align*} 
\delta_1 &:= a_h(\Pi_hu,v_h) + (\Delta u,\sfp(v_h)), \\
\delta_2 &:= (\Delta u,\sfp(v_h)-\mathcal{C}_h\sfp(v_h)).
\end{align*}
Using standard arguments shows that
\begin{align*}
|\delta_1| &\lesssim \left(\|J_h u - u\|^2_{\DG} +\sum_{T \in \mathcal{T}_h} h_T \|\nabla(J_h u -  u) \cdot \bm n_T\|^2_{L^2(\partial T)}\right)^{1/2} \|v_h\|_{V_h} \\
&\lesssim h \|v_h\|_{V_h}. 
\end{align*}
Moreover, invoking~\eqref{eq:approx_H1_Ch} and the stability of $\sfp$ gives
\[
|\delta_2| \lesssim \|\sfp(v_h)-\mathcal{C}_h\sfp(v_h)\|_{L^2(\Omega)} \lesssim
h\|\sfp(v_h)\|_{\DG} \lesssim h\|v_h\|_{V_h}.
\]
Combining the two bounds completes the proof.
\end{proof}

Everything is now in place to infer from Theorem~\ref{thm:general_error} the error estimate for the proximal H-IP method.

\begin{corollary}[Error estimate for proximal H-IP]\label{cor:errorHIP}
Assume that $u \in H^2(\Omega) \cap H^1_0(\Omega)$ and $\phi \in H^2(\Omega)$ with $\phi \vert_{\partial \Omega} \in \mathbb{P}^1(\mathcal{F}^\partial_h)$. Then, the proximal H-IP iterates generated  by \Cref{alg:main_alg_discrete} satisfy the following error estimate: For all $m \geq 1$, 
\begin{align} \label{eq:error_H-IP}
\|\sfp(u_h^m) - u\|^2_{\DG}  + \|\lambda_h^m - \lambda \|_{H^{-1}(\Omega)}^2 + \|o_h^m - u\|_{L^2(\Omega)}^2 
& \lesssim \frac{1}{\sum_{k=1}^m \alpha_k} + h^2. 
\end{align}
\end{corollary}

\begin{remark}[Continuous facet polynomials for H-IP]
\label{rmk:hip}
An alternative formulation of the H-IP method uses globally continuous facet unknowns
in the space $V_{\mathcal{F}_h}:= \mathbb{P}^r(\mathcal{F}_h) \cap C^0(\mathcal{F}_h),$ $r\ge1$.
This choice reduces the number of global degrees of freedom, and was first studied and analyzed in \cite{cockburn2009,guzey2007}. Similar arguments can be used to show that this method leads to the same error estimate as the one from Corollary~\ref{cor:errorHIP}.
\end{remark} 

\subsection{Proximal HHO}\label{sec:HHO}

Recall from Section~\ref{sec:hybrid_intro} that the discrete spaces are defined as
$V_h := V_{\mathcal{T}_h}  \times V_{\mathcal F_h} = \mathbb{P}^\ell (\mathcal{T}_h) \times \mathbb{P}^r(\mathcal{F}_h)$, with $\ell,  r \in \{0,1\}, \,\, \ell \leq r$, and $W_h := \mathbb{P}^0(\mathcal{T}_h)$. As for the H-IP method, a generic element of $V_h$ is denoted by $v_h = (v_{\mathcal{T}_h}, v_{\mathcal{F}_h}) \in V_h$, and the restriction of $v_h$ to $T \in \mathcal{T}_h$ is denoted by $(v_T, v_{\partial T})=(v_{\mathcal{T}_h} \vert_T, v_{\mathcal{F}_h} \vert_{\partial T}) \in 
V_T^{\ell,r}:= \mathbb{P}^\ell(T) \times \mathbb{P}^r(\mathcal{F}_T)$. 
The space $V_h$ is equipped with the norm $\|{\cdot}\|_{V_h}$ defined in~\eqref{eq:HIPnorm}. 

The devising of HHO methods is performed elementwise. For every element $T\in\mathcal{T}_h$, the local reconstruction operator $\mathsf{R}: V_T^{\ell, r} \to \mathbb{P}^{r+1}(T)$ is defined such that for all $(v_T,v_{\partial T}) \in V_T^{\ell,r}$,
\begin{equation} \label{eq:reconstruction}
(\nabla \mathsf{R}(v_T,v_{\partial T}), \nabla q)_{L^2(T)} := (\nabla v_T, \nabla q)_{L^2(T)} - ( v_T - v_{\partial T}, \nabla q \cdot \mathbf{n}_T )_{L^2(\partial T)} \quad \forall q \in \mathbb{P}^{r+1}(T),
\end{equation}
with the condition $(\mathsf{R}(v_T,v_{\partial T}) - v_T, 1)_{L^2(T)} := 0$.
Moreover, the local stabilization operator $\mathsf{S}: V_T^{\ell, r} \to \mathbb{P}^r(\mathcal{F}_T)$ is defined as
\begin{equation}
\mathsf{S}(v_T,v_{\partial T}) := \Pi_{\partial T}^r (v_T|_{\partial T} - v_{\partial T} + (I - \Pi_T^\ell) \mathsf{R}(v_T,v_{\partial T})|_{\partial T}),
\end{equation}
where $\Pi_{\partial T}^r$ and $\Pi_T^\ell$ are the $L^2$-orthogonal projections  onto $\mathbb{P}^r(\mathcal{F}_T)$ and $\mathbb{P}^\ell(T)$, respectively.
The local bilinear form $a_T$ is defined on $V_T^{\ell, r} \times V_T^{\ell,r}$ as
\begin{align}
a_T((v_T,v_{\partial T}),(w_T,w_{\partial T})) := {}& (\nabla \mathsf{R}(v_T,v_{\partial T}), \nabla \mathsf{R}(w_T,w_{\partial T}))_{L^2(T)}  \\
&+ h_T^{-1} (\mathsf{S}(v_T,v_{\partial T}), \mathsf{S}(w_T,w_{\partial T}) )_{L^2(\partial T)}.\nonumber
\end{align}
The global bilinear form $a_h$ is defined on $V_h \times V_h$ as
\begin{equation}\label{eq:hho_form}
a_h(v_h,w_h) := \sum_{T \in \mathcal{T}_h} a_T((v_T,v_{\partial T}),(w_T,w_{\partial T})).
\end{equation}
It is convenient to define the global reconstruction operator $\mathsf{R}_h:V_h\rightarrow \mathbb{P}^{r+1}(\mathcal{T}_h)$ so that $\mathsf{R}_h(v_h)|_T:=\mathsf{R}(v_T,v_{\partial T})$ for all $T\in\mathcal{T}_h$ and all $v_h\in V_h$.
The coercivity and continuity of the discrete bilinear form $a_h$ follows from the following property (see \cite{DiPEL:14}):
\begin{align}\label{eq:coercivity_HHO} 
\|v_h\|_{V_h}^2 \lesssim a_h(v_h,v_h) \lesssim \|v_h\|_{V_h}^2, \qquad \forall v_h\in V_h.
\end{align}
Hence, properties \eqref{eq:coercivity_a} and \eqref{eq:continuity_a} hold true. 

The HHO method offers two salient advantages with respect to the H-IP method owing to the use of the reconstruction operator. First, it avoids the need to consider a large enough penalty parameter $\sigma$. Second, it delivers higher-order decay rates on the consistency error (see Lemma~\ref{lemma:consistencyerror_HHO} below). For instance, the lowest-order HHO setting with $\ell=r=0$ leads to the same error estimates as the choice $\ell=r=1$ for H-IP. It should be noticed though, that the analysis of the proximal HHO method using the above unified framework is only possible for $\ell=1$. In the case $\ell=0$ and $r\in\{0,1\}$, a different (and simpler) analysis route is available. It is inspired from \cite{Cicuttin2020} and exposed at the end of this section.  

The linear map $\sfp:V_h\to \mathbb{P}^\ell(\mathcal{T}_h)$ is such that $\sfp(v_h)=v_{\mathcal{T}_h}$ for all $v_h = (v_{\mathcal{T}_h}, v_{\mathcal{F}_h}) \in V_h$. Moreover, the linear map 
$\Pi_h: H^1(\mathcal{T}_h) \rightarrow V_h$ is defined as
\begin{equation} 
\label{eq:Pi_h_hybrid_HHO}
\Pi_h v := (\Pi^\ell_{\mathcal{T}_h} v, \Pi^r_{\mathcal{F}_h} \{ v \} ), \qquad \forall v\in H^1(\mathcal{T}_h).
\end{equation}
Proceeding as in the proof of Lemma~\ref{lem:H-IP_basic} readily shows the following.

\begin{lemma}[Fulfillment of Assumption~\ref{ass:basic}]
Properties~\eqref{eq:bnd_sfp} and~\eqref{eq:Pihprop} hold true.
\end{lemma}

Since the proximal HHO methods satisfy Assumption~\ref{ass:basic}, all the results from Section~\ref{sec:basic_pty} (existence and uniqueness of iterates, energy dissipation) and Section~\ref{sec:min_reg} (error estimates with minimal regularity) are applicable to them. The error analysis with additional regularity per the above unified framework is only applicable, though, when $\ell=1$. Indeed, since $J_h$ is the $L^2$-orthogonal projection onto $V_{\mathcal{T}_h}:= \mathbb{P}^\ell(\mathcal{T}_h)$, Assumption~\ref{asm:Pih} is satisfied only in this case. Thus, we distinguish the case $\ell=1$ and $\ell=0$ in what follows. 

\subsubsection{Error estimate for $\ell=1$ and $r=1$}

Proceeding as in the proof of Lemma~\ref{lemma:consistency_HIP} shows that Assumption~\ref{asm:consistency} holds true. (This result is actually true regardless of the value 
of $\ell$.) Therefore, everything is in place to infer from Theorem~\ref{thm:general_error} the error estimate for the proximal HHO method with $\ell=r=1$.

\begin{corollary}[Error estimate for HHO, $\ell=r=1$]\label{cor:error_HHO_l=1}
Assume that $u \in H^2(\Omega) \cap H^1_0(\Omega)$ and $\phi \in H^2(\Omega)$ with $\phi \vert_{\partial \Omega} \in \mathbb{P}^1(\mathcal{F}^\partial_h)$. Set $\ell=r=1$ in the HHO method. Then, the proximal HHO iterates generated by \Cref{alg:main_alg_discrete} satisfy the error estimate~\eqref{eq:error_H-IP}. 
\end{corollary}

\begin{remark}[Reconstructed solution]
As highlighted in~\eqref{eq:elliptic_proj_hho} below, one expects higher-order convergence rates for the reconstructed solution $\mathsf{R}_h(u_h^m)$. Our current analysis does not leverage this property. The main difficulty stems from bounding the feasibility error $e_K(u_h^m)$. 
\end{remark}

\subsubsection{Error estimate for $\ell=0$ and $r\in\{0,1\}$}

In this case, the error analysis does not rely on the above unified framework, but follows a somewhat simpler route inspired from~\cite{Cicuttin2020}.  
Recall that for $\ell=0$, $J_h=\Pi_{\mathcal{T}_h}^0$ is the $L^2$-orthogonal projection onto $\mathbb{P}^0(\mathcal{T}_h)$.
We slightly modify the definition of the consistency error, which is now defined as 
the linear form $\tilde{\delta}_h(u) \in V_h'$ so that
\begin{equation}
\langle \tilde{\delta}_h(u), v_h \rangle_{V_h' , V_h} := a_h(\Pi_h u,v_h)  + (\Delta u, \sfp(v_h)), \quad \forall v_h \in V_h. \label{eq:tilde_delta_hho}
\end{equation}

The following bound on the consistency error is classical in the analysis of HHO methods.
We refer the reader to \cite{DiPEL:14} and \cite{CiErPi:21} for more details.

\begin{lemma}[Bound on consistency error]  \label[lemma]{lemma:consistencyerror_HHO}
Let $s \in (\frac12,r+1]$. Assume that $u \in H^{1+s}(\Omega) \cap H^1_0(\Omega)$. Then, 
\begin{equation}
\| \tilde{\delta}_h(u) \| _{V_h'} \lesssim  h^s |u|_{H^{1+s}(\Omega)}. 
\end{equation}
\end{lemma}

We notice that $s=r+1$ is possible, i.e., we recover a first-order estimate on the consistency error for $r=0$, $\| \tilde{\delta}_h(u) \| _{V_h'} \lesssim  h |u|_{H^2(\Omega)}$, and a second-order estimate for $r=1$, $\| \tilde{\delta}_h(u) \| _{V_h'} \lesssim  h^2 |u|_{H^3(\Omega)}$. While the assumption $u\in H^3(\Omega)$ goes beyond what is expected for the present unilateral Poisson obstacle problem, we notice that we can still exploit the regularity $u\in H^{\frac52-\epsilon}(\Omega)$, $\epsilon>0$. The improvement in the bound on the consistency error is rooted in the use of the reconstruction operator and the fact that the composition $\mathsf{R}_h \circ \Pi_h : H^1(\mathcal{T}_h) \rightarrow \mathbb{P}^{r+1}(\mathcal{T}_h)$ is the piecewise elliptic projection (preserving piecewise mean values). In particular, this implies that 
\begin{align} \label{eq:elliptic_proj_hho}
\sum_{T\in\mathcal{T}_h} \|\nabla (u -\mathsf{R}_h(\Pi_h u)) \|_{L^2(T)}^2 = 
\inf_{v \in \mathbb{P}^{r+1}(\mathcal{T}_h)} \sum_{T\in\mathcal{T}_h} \|\nabla(u-v)\|_{L^2(T)}^2.
\end{align}

\begin{theorem}[Error estimate for HHO, $\ell=0$, $r\in\{0,1\}$]\label{thm:error_HHO_l=0} 
\textup{(i)} Case $r=0$. Assume that $u \in H^2(\Omega) \cap H^1_0(\Omega)$ and that there is $\tau \in (0,1)$ so that $u - \phi \in W^{2,\frac{n}{\tau}}(\Omega)$ and $\lambda= -\Delta u - f \in W^{\tau , 1}(\Omega)$ (recall that $n\in\{2,3\}$ is the space dimension). 
Then the proximal HHO iterates generated by \Cref{alg:main_alg_discrete} satisfy the following error estimate: For all $m\geq1$, 
\begin{subequations} \label{eq:HHO_0}
\begin{align} \label{eq:hh000}
\sum_{T \in \mathcal{T}_h} \|\nabla (\mathsf{R}_h(u_h^m) - u)\|_{L^2(T)}^2  + \| u_h^m - \Pi_h u\|_{V_h}^2  \lesssim \frac{1}{\sum_{k=1}^m \alpha_k} + h^2. 
\end{align} 
\textup{(ii)} Case $r =1$. Assume that there is $\epsilon \in (0,\frac12]$ so that 
$u \in H^{\frac52 - \epsilon}(\Omega)$, $u - \phi \in W^{2 - \frac{\epsilon}{2} \frac{n-2}{n-1}, \frac{2}{\epsilon}(n-1)}(\Omega)$, and $\lambda \in W^{1-\epsilon,1}(\Omega)$. 
Then, the proximal HHO iterates generated by \Cref{alg:main_alg_discrete} satisfy the following error estimate: For all $m\geq1$, 
\begin{equation}\label{eq:HHO_error_reconstructed_r1}
 \sum_{T \in \mathcal{T}_h} \|\nabla (\mathsf{R}_h(u_h^m) - u)\|_{L^2(T)}^2  + \| u_h^m - \Pi_h u\|_{V_h}^2  \lesssim \frac{1}{\sum_{k=1}^m \alpha_k} + h^{3 - 2 \epsilon}. 
\end{equation} 
\end{subequations}
\textup{(iii)} In both cases for $r$, the following holds:
\begin{equation} \label{eq:lambda_o_HHO}
\|\lambda_h^m - \lambda \|_{H^{-1}(\Omega)}^2 + \|o_h^m - u\|_{L^2(\Omega)}^2 
\lesssim \frac{1}{\sum_{k=1}^m \alpha_k} + h^2. 
\end{equation}
\end{theorem}

\begin{proof} 
We first prove the bounds~\eqref{eq:HHO_0} on the primal variable. 
Our starting point is~\eqref{eq:inter3} and~\eqref{eq:energy_expansion_after_summing} from the proof of \Cref{thm:bestapprox}. Combining these two equations and invoking the coercivity property from~\eqref{eq:coercivity_HHO} gives 
\begin{align*} 
\| u_h^m - \Pi_hu\|_{V_h}^2
& \lesssim \frac{1}{\sum_{k=1}^m \alpha_k} + a_h(\Pi_hu ,  \Pi_h u - u_h^m) - (f,\Pi_{\mathcal{T}_h}^0 u  - u_{\mathcal{T}_h}^m )  \\ 
& = \frac{1}{\sum_{k=1}^m \alpha_k} + \langle \tilde \delta_h(u), \Pi_h u - u_h^m\rangle_{V_h', V_h} +  (\lambda,\Pi_{\mathcal{T}_h}^0 u  - u_{\mathcal{T}_h}^m ), 
\end{align*}
where we used the definition~\eqref{eq:tilde_delta_hho} of the consistency error and the fact that $\lambda$ is an integrable function in $\Omega$ owing to our regularity assumption on $u$.  
Invoking Young's inequality for the consistency error gives
\[
\| u_h^m - \Pi_hu\|_{V_h}^2 \lesssim \frac{1}{\sum_{k=1}^m \alpha_k} + \|\tilde \delta_h(u)\|_{V_h'}^2 +  (\lambda,\Pi_{\mathcal{T}_h}^0 u  - u_{\mathcal{T}_h}^m ).
\]
Let us consider the third term on the right-hand side.
Owing to the complementarity conditions and the regularity assumptions, we infer that $\lambda \geq 0$, $\lambda (u -\phi) = 0$ a.e. in $\Omega$, and $\lambda = 0$ in $ \{ x \in \Omega : u(x) > \phi(x) \}$. Following \cite{Cicuttin2020}, we split the mesh into three subsets (non-contact, contact, free) independently of the discrete solution: 
$\mathcal{T}_h^{\textsc{nc}}  := \{ T \in \mathcal{T}_h \mid  u > \phi \text{ a.e. in } T \},$ 
$\mathcal{T}_h^{\textsc{c}} := \{ T \in \mathcal{T}_h \mid  u = \phi \text{ a.e. in } T \},$ and 
$\mathcal{T}_h^{\textsc{f}}   :=  \mathcal{T}_h \setminus (\mathcal{T}_h^{\textsc{nc}} \cup \mathcal{T}_h^{\textsc{c}}).$ 
For all $T \in \mathcal{T}_h^{\textsc{nc}}$, we have $\lambda = 0$. For all $T \in \mathcal{T}_h^{\textsc{c}}$, we have 
\begin{align}
(\lambda, \Pi_T^0 u - u_T^m)_T = (\lambda , \Pi_T^0 \phi  - u_{T}^m)_T \leq 0, 
\end{align}
owing to Lemma~\ref{lemma:Kh}, recalling that $u_T^m \in \mathbb{P}^0(T)$ and that $\lambda\ge0$. (Notice in passing that the last inequality is valid more generally on all $T\in\mathcal{T}_h$.) Therefore, the above error estimate becomes
\begin{align}
\|u_h^m - \Pi_h u\|^2_{V_h} \lesssim \frac{1}{\sum_{k=1}^m \alpha_k}  + \| \tilde \delta_h(u)\|_{V_h'}^2 +\sum_{T \in \mathcal{T}_h^{\textsc{f}}} (\lambda, \Pi_T^0 u - u_T^m)_T.
\end{align}
From here, we use \Cref{lemma:consistencyerror_HHO} to bound the second term and \cite[Lemma 6]{Cicuttin2020} to bound the third term (notice that the latter bound only uses the property $(\lambda , \Pi_T^0 \phi  - u_{T}^m)_T \leq 0$ satisfied by $u_T^m$). This provides the bound on $\|u_h^m - \Pi_h u\|^2_{V_h}$ stated in~\eqref{eq:HHO_0}. To obtain the bound on the reconstructed solution, we observe that
\begin{multline*}
\sum_{T \in \mathcal{T}_h} \|\nabla (\mathsf{R}_h(u_h^m) - u)\|_{L^2(T)}^2
\\ \le 2 \sum_{T \in \mathcal{T}_h} \Big\{ \|\nabla \mathsf{R}_h(u_h^m - \Pi_hu)\|_{L^2(T)}^2
+ \|\nabla (\mathsf{R}_h\Pi_hu - u)\|_{L^2(T)}^2 \Big\}.
\end{multline*}
The first term on the right-hand side is estimated using the second bound in \eqref{eq:coercivity_HHO} and the above bound on $\|u_h^m - \Pi_h u\|^2_{V_h}$, whereas the second term is estimated using~\eqref{eq:elliptic_proj_hho}.  

We now prove~\eqref{eq:lambda_o_HHO}. We observe that
\begin{align}\label{eq:t1t2t3}
\langle \lambda_h^m - \lambda, v \rangle = &{}  a_h(u_h^m - \Pi_h u, \Pi_h v) + \langle \tilde \delta_h (u) , \Pi_h v \rangle_{V_h',V_h} \\ & - (\lambda,v-\Pi^0_{\mathcal{T}_h}v), \quad \forall v \in H^1_0(\Omega).  \nonumber
\end{align}
Let $T_1$, $T_2$, $T_3$ denote the three terms on the right-hand side. 
Since $\|\Pi_hv\|_{V_h} \lesssim \|\nabla v\|_{L^2(\Omega)}$, we readily infer that
\[
|T_1+T_2|\lesssim \big( \|u_h^m - \Pi_h u\|_{V_h} + \|\tilde{\delta}_h(u)\|_{V_h'} \big) \|\nabla v\|_{L^2(\Omega)}.
\]
Moreover, we have $|T_3|\lesssim h$. This completes the proof of the estimate on $\|\lambda_h^m-\lambda\|_{H^{-1}(\Omega)}^2$. 
Finally, the estimate on the bound-preserving approximation $o_h^m$ follows from Lemma~\ref{lem:conv_ohm} and the broken Poincar\'e inequality.
\end{proof}

\begin{remark}[Improved convergence on Lagrange multiplier]
For $r=1$, if we further assume that $\lambda\in H^{\frac12-\epsilon}(\Omega)$, the error estimate on $\lambda_h$ can be improved to
\begin{equation*}
\|\lambda - \lambda_h\|_{H^{-1}(\Omega)}^2  \lesssim \frac{1}{\sum_{k=1}^m \alpha_k} + h^{3 - 2 \epsilon}. 
\end{equation*} 
This follows from \eqref{eq:t1t2t3}, the subsequent bound on $T_1+ T_2$, and the following improved estimate on $T_3$:
\begin{equation*}
T_3=(\lambda-\Pi^0_{\mathcal{T}_h}\lambda,v-\Pi^0_{\mathcal{T}_h}v)\lesssim h^{3/2-\epsilon}\|\nabla v\|_{L^2(\Omega)},
\end{equation*}
where we exploited that $\lambda\in H^{\frac12-\epsilon}(\Omega)$ to gain the additional factor $h^{\frac12-\epsilon}$.
\end{remark}

\section{Numerical Experiment} \label{sec:num}
Consider the following obstacle in $\Omega := (-1,1)^2:$
\begin{equation*}
   \phi := \begin{cases}
        \sqrt{ 1/4 - r^2} & \mathrm{if } \quad r \leq 9/20, \\ 
        \varphi (r) & \mathrm{otherwise},
    \end{cases} \qquad \text{where } r := \sqrt{x^2 + y^2}\,.
\end{equation*}
In the above, $\varphi(r)$ is the unique $C^1$ linear extension of $r \mapsto \sqrt{1/4 - r^2}$ for $r > 9/20$.
The exact solution with $f\equiv0$ is given by 
\begin{equation*}
        u := \begin{cases}
    Q  \ln \sqrt{x^2+y^2} &\mathrm{if } \quad \sqrt{x^2 +y^2} >  a, \\ 
\phi & \mathrm{otherwise},
    \end{cases} 
\end{equation*}
where $a:= \exp(W_{-1}(-1/(2e^2))/2 + 1)\approx 0.34898$ and $W_{-1}$ denotes the $-1$-branch of the Lambert W-function, and $Q := \sqrt{1/4-a^2}/\ln a$. The solution $u\in H^{5/2-\epsilon}(\Omega)$ and the corresponding multiplier $\lambda\in H^{1/2-\epsilon}(\Omega)$ are depicted in \Cref{fig:ex1-sol}.
\begin{figure}
    \centering
    \begin{tabular}{cc}
    \includegraphics[trim={2cm, 0cm, 2cm, 4cm}, clip, width=0.4\textwidth]{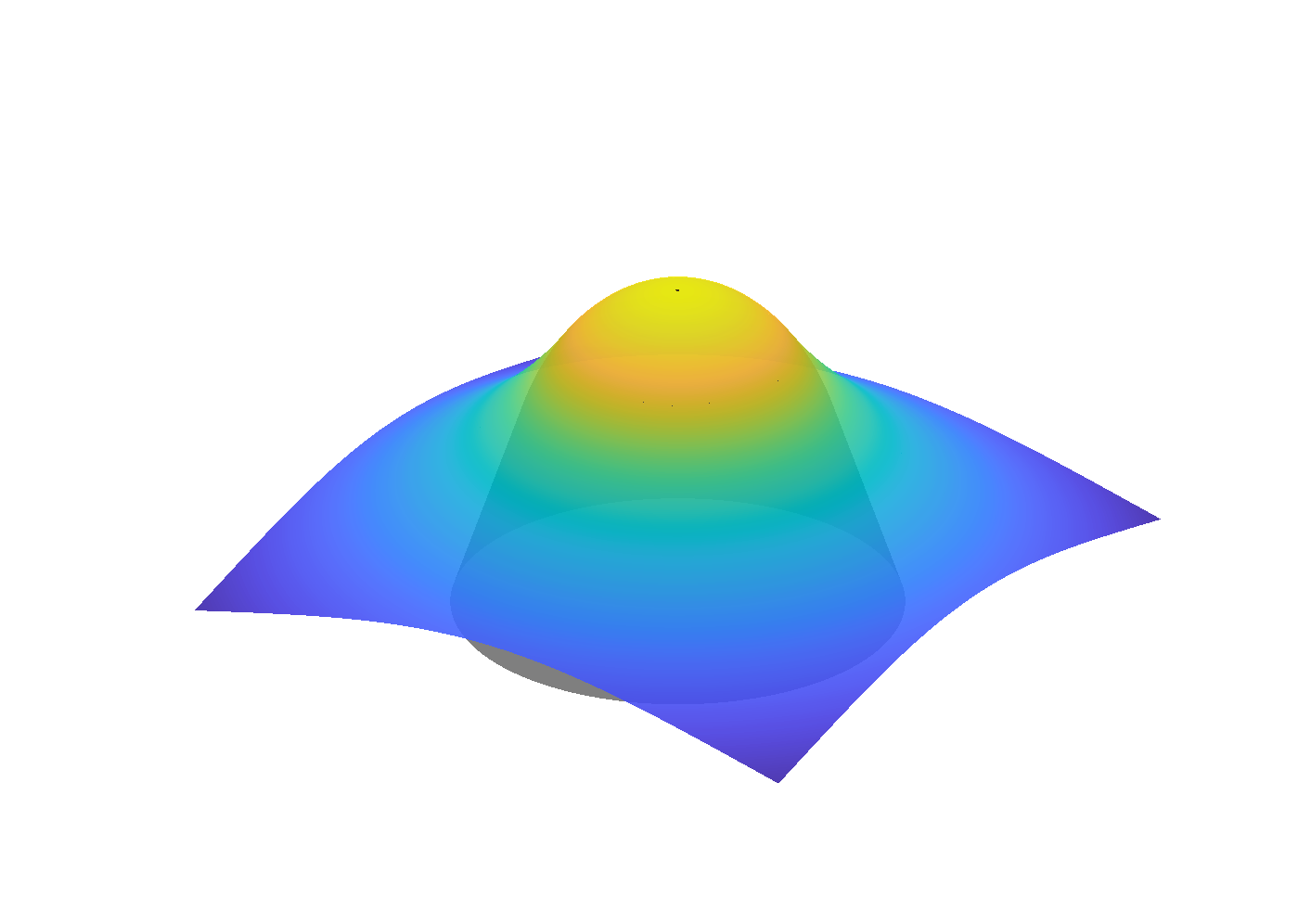}&
    \includegraphics[width=0.4\textwidth]{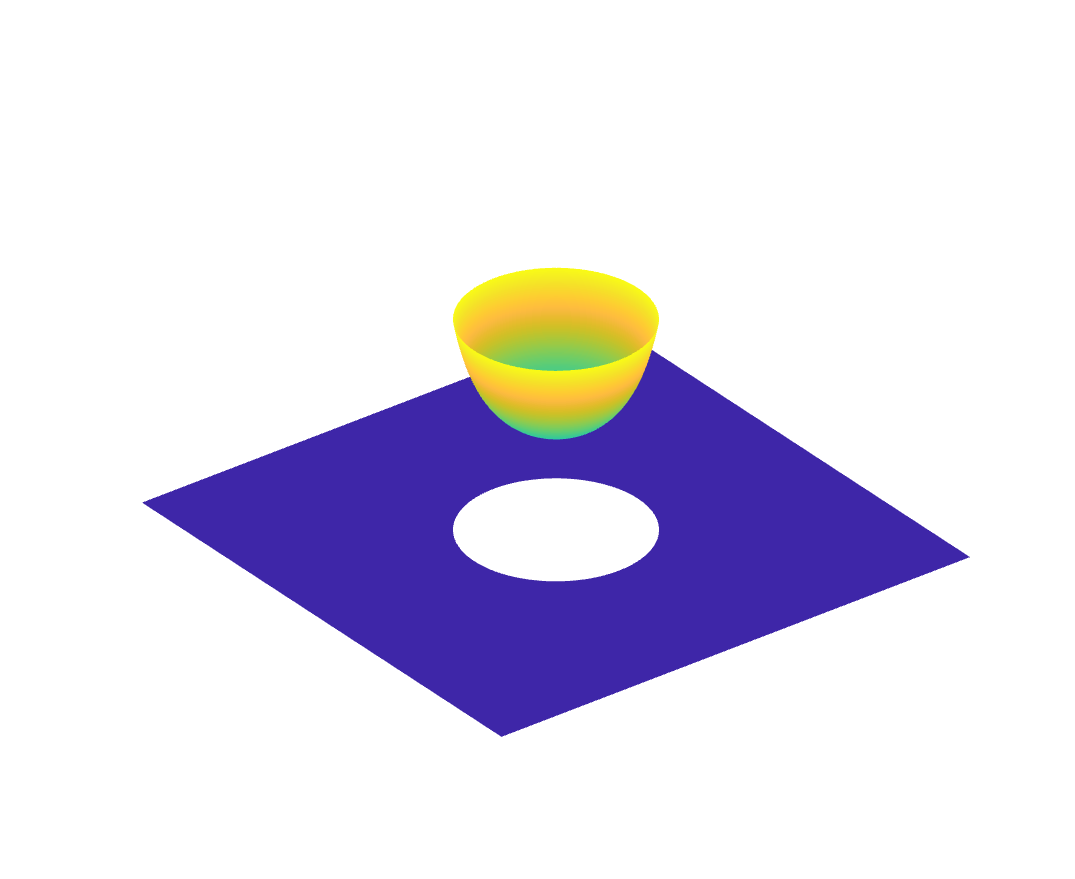}\\
    (a)&(b)
    \end{tabular}
    \caption{(a) The solution and the obstacle and (b) the corresponding Lagrange multiplier $\lambda$.\label{fig:ex1-sol}}
\end{figure}

To verify the convergence rate in \Cref{thm:general_error}, we consider a sequence of uniform triangular meshes with side length $2^{-3},...,2^{-7}$.
For all the relevant methods (SIPG, EG, H-IP and HHO($\ell=r=1$)), we use the polynomial degree $k=1$.
The convergence history is reported in \Cref{fig:ex1-conv}.
In addition to the nonconforming methods introduced in this work, we also included a stable conforming pair (denoted by MINI, \cite{keith2023proximal}),
\begin{equation*}
    V_h\times W_h = \Big((H^1_0(\Omega)\cap \mathbb{P}^1(\mathcal{T}_h))\oplus \mathbb{B}_3(\mathcal{T}_h)\Big)\times \mathbb{P}^0(\mathcal{T}_h),
\end{equation*}
where $\mathbb{B}_3(\mathcal{T}_h)$ consists of the cubic bubble functions on $\mathcal{T}_h$.

All methods show optimal convergence in $H^1$-norm as expected from \Cref{thm:general_error}.
The solution also converges optimally in the $L^2$-norm and
the approximate multiplier $\lambda_h$ converges in the $L^2$-norm with a rate $\mathcal{O}(h^{1/2})$.
While all methods show a similar tendency, the HHO method provides the best accuracy, probably due to the use of the higher-order reconstruction operator.

\begin{figure}
    \centering
    \begin{tabular}{ccc}
    \includegraphics[width=0.3\textwidth]{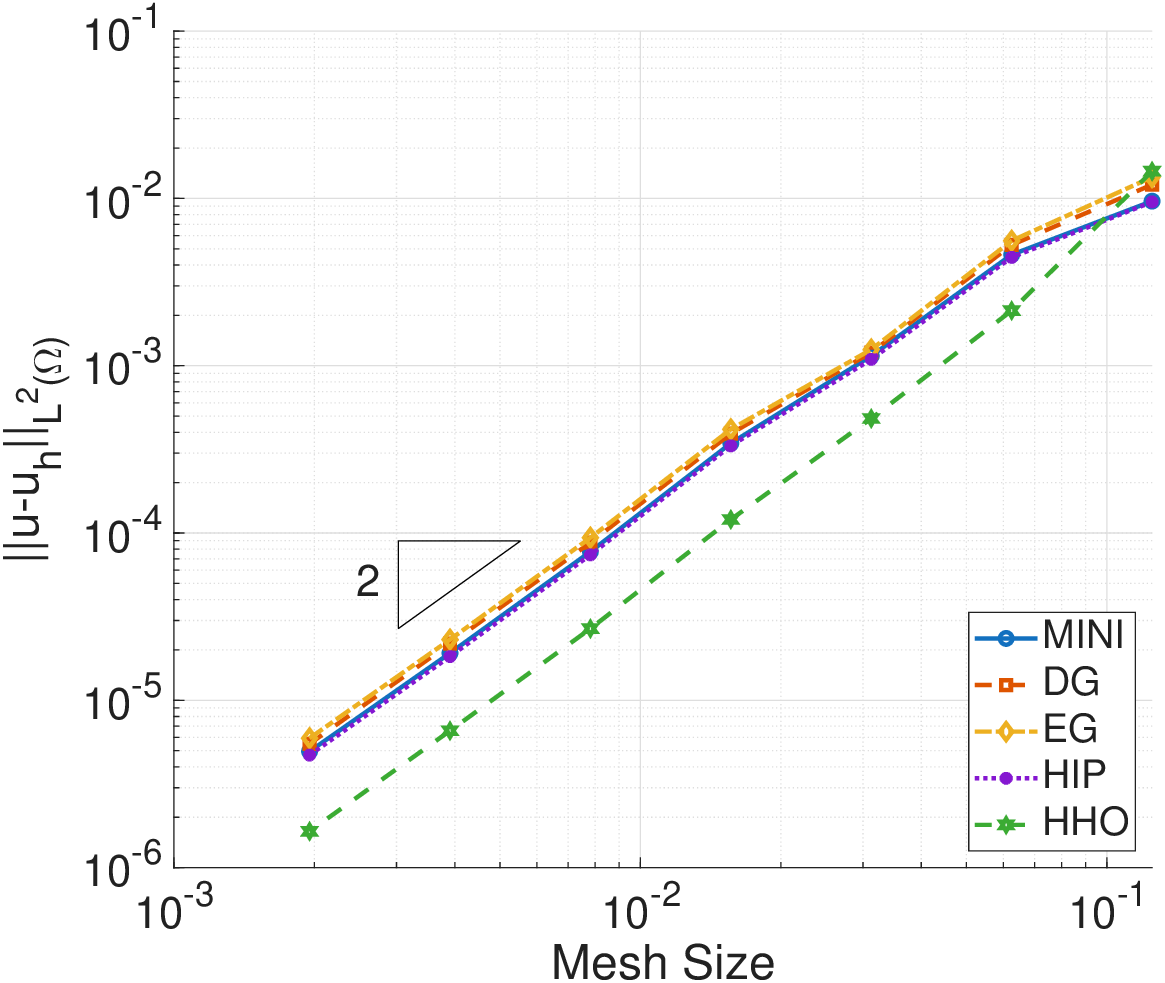}&
    \includegraphics[width=0.3\textwidth]{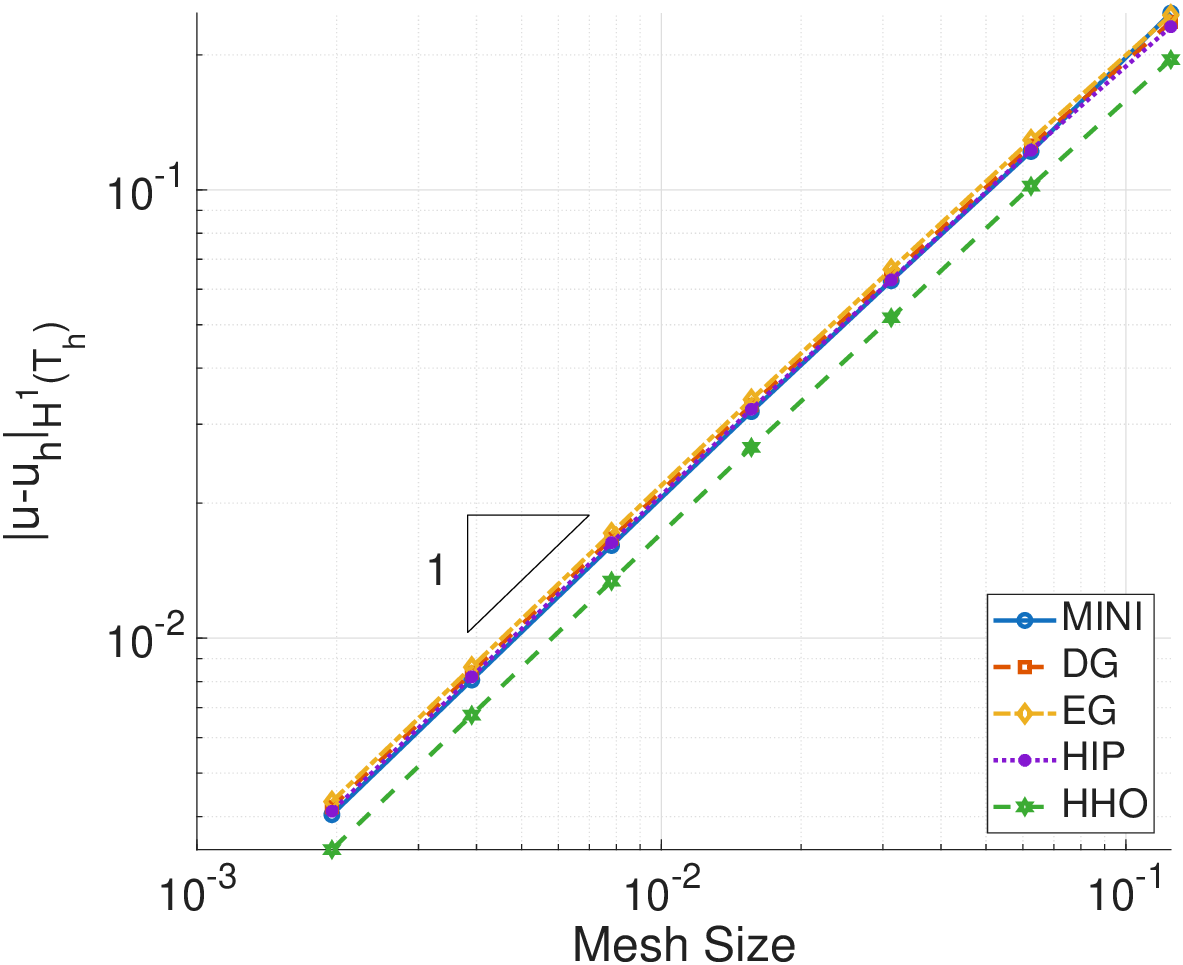}&
    \includegraphics[width=0.3\textwidth]{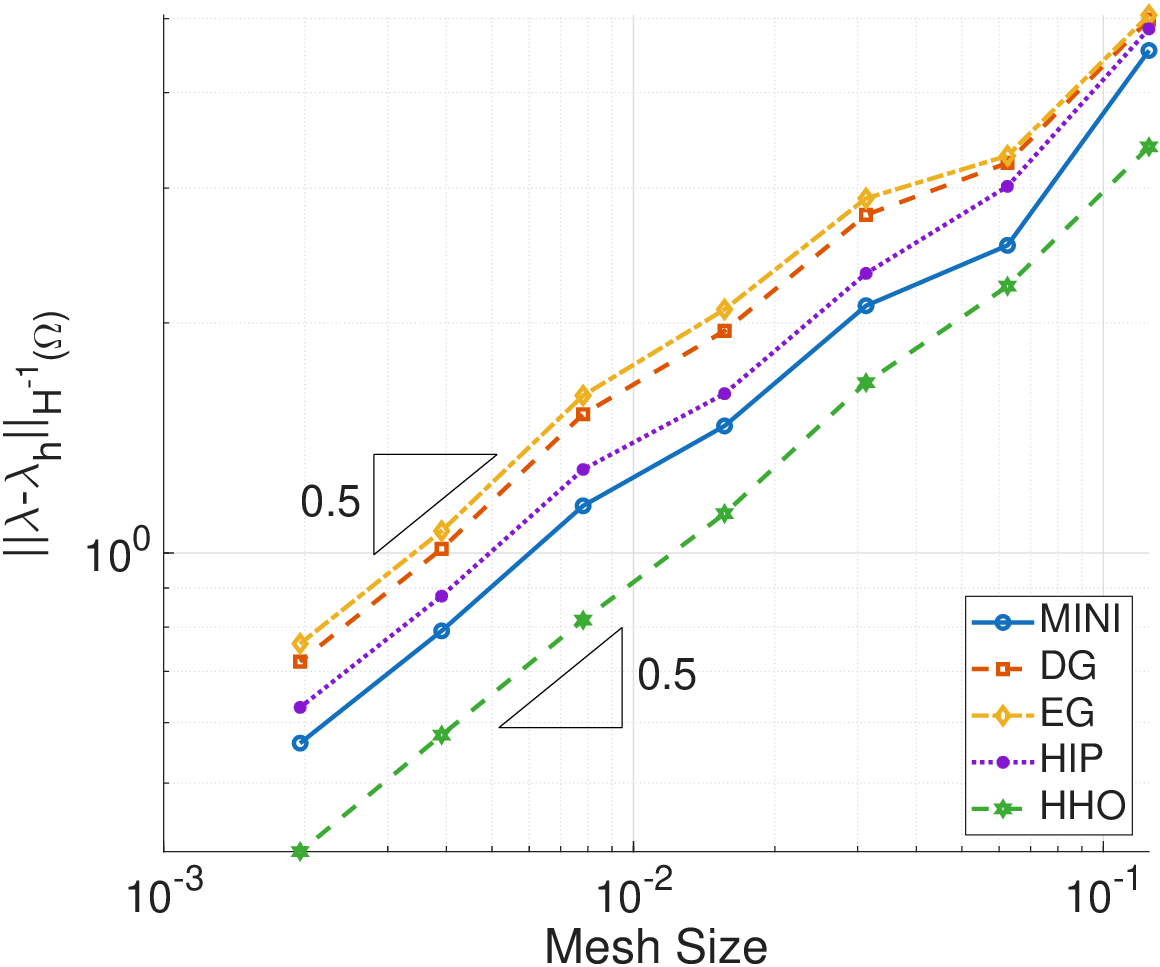}\\
    (a)&(b)&(c)\\
    
    \end{tabular}
    \caption{Convergence history of (a) $\|u-u_h\|_{L^2(\Omega)}$; (b) $|u-u_h|_{H^1(\mathcal{T}_h)}$; (c) $\|\lambda-\lambda_h\|_{L^2(\Omega)}$.\label{fig:ex1-conv}}
\end{figure}

The second experiment examines HHO more closely.
We test the three proximal HHO methods, corresponding to $(\ell,r)=(0,0),(0,1),(1,1)$, and report errors for the reconstructed solution $\mathsf{R}(u_h)\in\mathbb{P}^{2}(\mathcal{T}_h)$ defined in \eqref{eq:reconstruction}.
As shown in \Cref{fig:ex1-HHO-conv}, the variant with $\ell=0$ and $r=1$ yields decay rates of order $\frac32$, as predicted by Theorem~\ref{thm:error_HHO_l=0} (Case (ii)). The same rate is observed for the variant with $\ell=r=1$, although not predicted by the theory so far, whereas the variant with $\ell=r=0$ leads to first-order decay rates, as predicted by Theorem~\ref{thm:error_HHO_l=0} (Case (i)).

\begin{figure}
    \centering
    \begin{tabular}{cc}
    \includegraphics[width=0.4\textwidth]{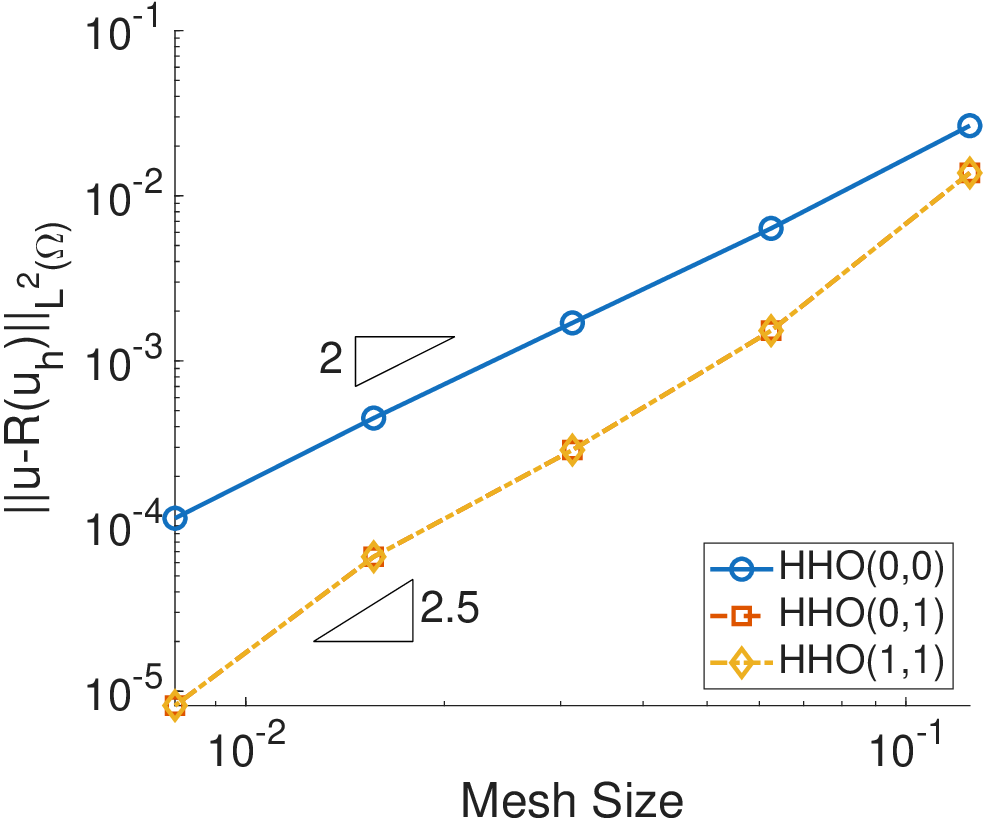}&
    \includegraphics[width=0.4\textwidth]{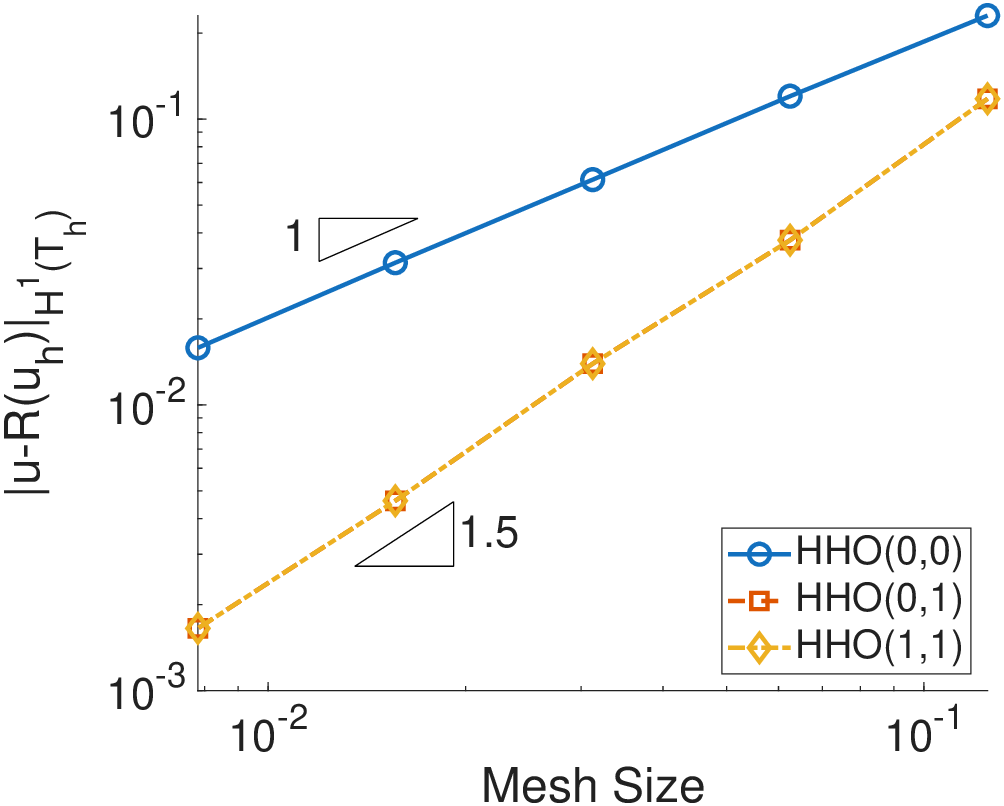}\\
    (a)&(b)\\
    
    \end{tabular}
    \caption{Proximal HHO$(\ell , r)$ method: Convergence history of (a) $\|u-\mathsf{R}(u_h)\|_{L^2(\Omega)}$; (b) $|u-\mathsf{R}(u_h)|_{H^1(\mathcal{T}_h)}$.\label{fig:ex1-HHO-conv}}
\end{figure}


\section{Conclusion}\label{sec:conc}
We have introduced and analyzed a family of proximal discontinuous Galerkin methods for the unilateral Poisson obstacle problem, extending the PG framework developed so far essentially for conforming discretizations. We established best-approximation properties with minimal regularity assumptions. We also showed optimal error estimates in the energy norm under mild additional assumptions, demonstrating that the optimization error is decoupled from the spatial discretization error. A notable outcome is that the HHO method with piecewise constant cell unknowns and piecewise linear facet polynomials $(\ell = 0,\, r = 1)$ yields an energy error estimate of order $h^{3/2 - \epsilon}$ for the locally reconstructed solution. Numerical experiments confirm the theoretical convergence rates across all methods. 

\bibliographystyle{plain}
\bibliography{references}
\end{document}